\documentclass[12pt,a4paper,reqno,amsmath,amsthm,amssymb,amscd]{amsart}
\usepackage{pifont}

\usepackage{amscd,amsmath,amsthm,amssymb}
\usepackage{mathrsfs}
\usepackage{multicol}\multicolsep=0pt
\usepackage{pstricks,pst-node}
\usepackage[backref=page]{hyperref}
\usepackage[enableskew,vcentermath]{youngtab}

\usepackage[sort]{cite}

\def\crulefill{\leavevmode\leaders\hrule height 1pt\hfill\kern 0pt}
\long\def\QUERY#1{%
\leavevmode\newline%
\noindent$\star\star\star$\thinspace\textsf{Comment/Query}\crulefill\newline%
   \space #1\newline\hbox to 120mm{\crulefill}$\star\star\star$\newline}
\newtheorem{thm}{Theorem}[section]
\newtheorem{Lemma}[thm]{Lemma}
\newtheorem{Cor}[thm]{Corollary}
\newtheorem{Prop}[thm]{Proposition}

 \theoremstyle{definition}
\newtheorem{example}[thm]{Example}

\newtheorem{Defn}[thm]{Definition}

\newtheorem{rem}[thm]{Remark}
\newtheorem{rems}[thm]{Remarks}

\numberwithin{equation}{thm}

\makeatletter
\def\enumerate{\begingroup\ifnum\@enumdepth>3\@toodeep\else
      \advance\@enumdepth\@ne
      \edef\@enumctr{enum\romannumeral\the\@enumdepth}%
      \topsep\z@\parskip\z@
      \list{\csname label\@enumctr\endcsname}
        {\@nmbrlisttrue\let\@listctr\@enumctr
         \parsep\z@\itemsep\z@\topsep\z@
         \setcounter{\@enumctr}{0}
         \def\makelabel##1{\hss\llap{\rm ##1}}
       }\fi}

\makeatother

\let\bar=\overline
\let\epsilon=\varepsilon
\def\({\big(}
\def\){\big)}

\def\ZC{\mathcal Z}

\def\0{\underline{0}}

\def\ts{\tilde\s}

\def\ss{\mathsf s}
\def\ts{\mathsf t}

{\catcode`\|=\active
  \gdef\set#1{\mathinner{\lbrace\,{\mathcode`\|"8000%
                                   \let|\midvert #1}\,\rbrace}}
  \gdef\seT#1{\mathinner{\Big\lbrace\,{\mathcode`\|"8000%
                                   \let|\midverT #1}\,\Big\rbrace}}
}
\def\midvert{\egroup\mid\bgroup}
\def\midverT{\egroup\,\Big|\,\bgroup}

\def\Set[#1]#2|#3|{\Big\{\ #2\ \Big| \
           \vcenter{\hsize #1mm\centering #3}\Big\}}

\def\up{{\boldsymbol\upsilon}}
\def\bsq{{\boldsymbol q}}
\def\fS{{\mathfrak S}}\def\fT{{\mathfrak T}}
\def\fD{{\mathfrak D}}\def\fJ{{\mathfrak J}}
\def\fH{{\mathfrak H}}

\def\sA{{\mathcal A}}

\def\sE{{\mathcal E}}

\def\sH{{\mathcal H}}

\def\sL{{\mathcal L}}

\def\sR{{\mathcal R}}
\def\sS{{\mathcal S}}
\def\sT{{\mathcal T}}

\def\sZ{{\mathcal Z}}

\def\la{{\lambda}}
\def\lra{{\longrightarrow}}
\def\ro{{\text{\rm ro}}}
\def\co{{\text{\rm co}}}
\def\vph{{\varphi}}
\def\sfC{{\mathsf C}}
\def\ssS{{\mathsf S}}
\def\ssT{{\mathsf T}}

\def\bfT{{\mathbf T}}
\def\RS{\overset{\textrm{RS}}\longrightarrow}
\def\RSKs{\overset{\textrm{\rm RSKs}}\longrightarrow}
\def\sS{\mathcal S}
\def\bsS{{\boldsymbol{\mathcal S}}}
\def\bsH{{\boldsymbol{\mathcal H}}}
\def\bfT{{\mathbf T}}
\def\bsDe{{\boldsymbol{L}}}
\def\bsE{{\boldsymbol{\sE}}}
\def\bsT{{\boldsymbol{\fT}}}
\def\field{{F}}
\def\bfU{{\mathbf U}}
\def\wt{{\text{\rm wt}}}
\def\id{{\text{\rm id}}}

\begin{document}
\title{  Quantum Schur superalgebras and  Kazhdan--Lusztig combinatorics }
\author{Jie Du and Hebing Rui}
\address{J.D. School of Mathematics and Statistics,
The University of New South Wales, Sydney NSW 2052, Australia}
\email{j.du@unsw.edu.au}
\address{H.R. Department of Mathematics,  East China Normal
University, Shanghai, 200062, China} \email{hbrui@math.ecnu.edu.cn}
\thanks{Partially supported by ARC and National Natural Science Foundation of China (NSFC)}

\date{\today}

\begin{abstract}
We introduce the notion of quantum Schur (or $q$-Schur) superalgebras. These
algebras share certain nice properties with $q$-Schur algebras such as base change property, existence of canonical $\mathbb Z[v,v^{-1}]$-bases, and
the duality relation with quantum matrix superalgebra $\sA(m|n)$.  We also construct
a cellular $\mathbb Q(\up)$-basis and determine its associated cells, called super-cells,
in terms of a Robinson--Schensted--Knuth super-correspondence.
In this way, we classify all irreducible representations over $\mathbb Q(\up)$ via super-cell modules.
\end{abstract}

 \sloppy \maketitle

\section{Introduction}

The quantum Schur (or $q$-Schur) algebra is the key ingredient of a so-called quantum Schur--Weyl theory.
This theory investigates a three-level duality relation which includes: (1) quantum Schur-Weyl reciprocity for
the quantum enveloping algebra $\mathbf U(\mathfrak{gl}_n)$ and Hecke algebras $\boldsymbol\sH(\fS_r)$
via the tensor space $V_n^{\otimes r}$ --- a $\mathbf U(\mathfrak{gl}_n)$-$\boldsymbol\sH(\fS_r)$-bimodule;
(The algebras $\boldsymbol\sS(n,r):=\text{End}_{\boldsymbol\sH}(V_n^{\otimes r})$ are homomorphic images of
$\mathbf U(\mathfrak{gl}_n)$ and are called quantum Schur algebras.)
(2) certain category equivalences between categories of $\boldsymbol\sH$-modules and $\boldsymbol\sS(n,r)$-modules; (3) the realization and presentation problems in which quantum $\mathfrak{gl}_n$ is reconstructed via quantum Schur algebras as a vector space together with certain explicit multiplication formulas on basis elements, and quantum Schur algebras are presented by generators and relations.
We refer the reader to Parts 3 and 5 of \cite{DDPW} and the reference therein for a full account of the quantum Schur--Weyl theory, and to \cite{DDF} for the affine version of the theory.
Naturally, one expects a super version of the quantum Schur--Weyl theory.

Schur superalgebras and their quantum analogue have been investigated
in the context of (quantum) general linear Lie superalgebras or
supergroups; see, e.g., \cite{Muir}, \cite{BKu}, \cite{Don01}, \cite{Mit}.
For example,
Mitsuhashi \cite{Mit} has established (for a generic $q$) the super version of quantum Schur-Weyl reciprocity, and Brundan and Kujawa have investigated representations for Schur superalgebras and provided
a surprising application to the proof of Mullineux conjecture. Thus, like quantum Schur algebras, quantum Schur  superalgebras will play a decisive role in a super-version of the quantum Schur--Weyl theory.

  In this paper, we will investigate quantum Schur superalgebras in the
context of Hecke algebras and Kazhdan--Lusztig combinatorics. We
will first define a quantum Schur superalgebra as the endomorphism
superalgebra of certain signed $q$-permutation modules for Hecke
algebras of type $A$.  By introducing standard and canonical bases,
we establish a cell theory for quantum Schur superalgebras. Thus, a
super version of Robinson-Schensted-Knuth correspondence is
developed to get the cell decomposition and the classification of
(ordinary) irreducible representations.

We organize the paper as follows. After a brief review of Hecke algebras
and their Kazhdan--Lusztig combinatorics,  we discuss, as preparation,
some combinatorial facts, including a description of super-representatives
of double cosets and the Robinson--Schensted--Knuth (RSK) super-correspondence.
We introduce in \S5 the notion of quantum Schur superalgebra by using the $q$-analogues
 of the modules $M^{\la,\mu}$ given in \cite[1.2]{Ser} and prove that this is the same
 algebra as given in \cite{Mit}
defined by the tensor superspace. We construct an integral standard basis which is used to
 construct a Kazhdan--Lusztig type (or canonical) basis in \S6. In particular, we establish
  the base change property. In order to understand its representations over the field
   $\mathbb Q(\up)$, we further introduce another basis, a cellular type basis, over the field
   via the Kazhdan--Lusztig basis of the Hecke algebra. Thus, cell relations can be introduced and
cells modules form a complete set of non-isomorphic irreducible
modules. This result can be considered as a generalization of
Theorem~1.4 in \cite{KL79} to the super case.  Finally, we prove that quantum Schur
superalgebras are the linear dual of the homogeneous components of
the quantum general linear supergroup introduced by Manin \cite{Manin}.

{\it Throughout the paper, we make the following notational convention. }

Let $m,n$ be
nonnegative integers, not both zero. Let $\mathbb Z_2=\mathbb
Z/2\mathbb Z$ be the set of integers modulo 2. Fix the map
\begin{equation}\label{hatmap}
\hat{\ }:\{1,2,\ldots,m,m+1,\ldots, m+n\}\to\mathbb Z_2
\end{equation}
such that $\hat i=\begin{cases} 0,&\text{ if }1\le i\le m,\\
1,&\text{ if }m+1\le i\le m+n.\end{cases}$

Let $\ZC=\mathbb Z[\up, \up^{-1}]$
be the ring of Laurent polynomials in indeterminate $\up$. If $\sA$ denotes a $\sZ$-algebra,
we shall use the same letter of boldface $\boldsymbol\sA$ to denote the $\mathbb Q(\up)$-algebra
obtained by base change to $\mathbb Q(\up)$. In other words, $\boldsymbol\sA=\sA\otimes \mathbb Q(\up)$.

\vspace{.3cm}

\noindent
{\bf Acknowledgement.} The authors would like to thank Weiqiang Wang for making the reference \cite{Ser} available to us. The paper was written while the first author was taking a sabbatical leave from UNSW. He would like to thank East China Normal University, Universities of Mainz and Virginia
for their hospitality and financial support during the writing of the paper.

\section{Hecke algebras and their Kazhdan--Lusztig combinatorics}
  Assume for the moment that
$\sH=\sH(W)$ is the Hecke algebra associated with a {\it Coxeter system} $(W,S)$. Thus,
$\sH$ is an associative $\ZC$-algebra with basis $\{T_w\}_{w\in W}$
 subject to the relations (where
$\bsq=\up^2$)
\begin{equation}\label{heckea}\begin{cases}
 T_s^2=(\bsq-1)T_s+\bsq, & \text{for $s\in S$;}\\
T_y  T_w=  T_{yw}, &\text{$l(yw)=l(y)+l(w)$,} \\
\end{cases}
\end{equation}
where $l$ is the length function relative to $S$. Clearly, $\sH$ admits an
anti-involution $\tau:\sH\to\sH$ sending $T_w$ to $T_{w^{-1}}$. We
first briefly review the construction of the canonical (or
Kazhdan--Lusztig) bases of Hecke algebras.

 Let ${}^-: \mathcal H\rightarrow \mathcal H$ be the $\mathbb
Z$-linear  involution on $\mathcal H$ such that $\bar{\up}=\up^{-1}$
and $\bar {T_w}=T_{w^{-1}}^{-1}$. In \cite{KL79}, Kazhdan and
Lusztig showed that, for any $w\in W$, there is a unique element
$C_w\in \mathcal H$ such that $\bar{C_w}=C_w$ and
\begin{equation}\label{Cbasis}
C_w=\up^{-l(w)}\sum_{y\leq w} P_{y, w}(\up^2) T_y\end{equation} where
$\le $ is the Chevalley-Bruhat order on $W$ and $P_{y, w}$ is a
polynomial in $\bsq=\up^2$ with degree less than ${1\over
2}(l(w)-l(y)-1)$ for $y<w$ and $P_{w, w}=1$. Moreover, $\set{C_w\mid
w\in W}$ forms a free $\sZ$-basis of $\mathcal H$.

Let $\iota$ be the involution on $\mathcal H$ defined by setting
$$
\iota(\sum_{w\in W} a_w T_w)=\sum_{w\in W} \epsilon_w
\bar{a_{w}}\up^{-2l(w)} T_w,\,\,
\text{ where }\epsilon_w=(-1)^{l(w)}.$$ Write $B_w=\epsilon_w \iota(C_w)$. Then $
B_w=\sum_{y\le w} \epsilon_y \epsilon_w \up^{l(w)} \up^{-2l(y)} \bar
P_{y, w} T_y$.  Both $\{C_w\}_{ w\in W}$ and $\{B_w\}_{w\in
W}$ are called {\it canonical or Kazhdan-Lusztig bases}\footnote{In
\cite{KL79}, $C_w$ is denoted by $C'_w$, while $B_w$ is denoted by
$C_w$.}  for
$\mathcal H$.

For $x, y\in W$, let $\mu(y, w)$ be the coefficient of
$\bsq^{\frac12(l(w)-l(y)-1)}$ in $P_{y, w}$. The following formulae
are due to Kazhdan and Lusztig~\cite{KL79}.

For any $s\in S$ and $w\in W$,
\begin{equation}\label{leftcsimple}
 C_s C_x=\begin{cases} (\up+\up^{-1}) C_w, & \text{if
 $sw<w$},\\ C_{sw}+\sum_{\substack{y<x, sy<y\\ l(y)\not\equiv l(x) (2)\\}}    \mu(y, w)
 C_y, &\text{if $sw>w$}.\\
 \end{cases}\end{equation}
Here, $\leq$ denote the Bruhat ordering of $W$.

Canonical bases have important applications to representations of
Hecke algebras through the notion of cells.  Following \cite{KL79},
we define preorder $\le_L$ on $W$ by declaring that $x\le_L y$ if
there is a sequence $z_0=x, z_1, \dots, z_k=y$ such that  $C_{z_i} $
appears in the expression of $C_sC_{z_{i+1}}$  with non-zero
coefficient for some $s\in S$. Define $x\le_R y$ by declaring that
$x^{-1}\le_L y^{-1}$. Let $\le_{LR}$ be the preorder generated by
$\le_L $ and $\le_R$. The corresponding equivalence relations are
denoted by $\sim_L, \sim_R$ and $\sim_{LR}$. Call the
equivalence classes of $W$ with respect to $\sim_L, \sim_R$ and
$\sim_{LR}$, respectively,  \textit{left cells, right cells and two-sided
cells} of $W$.

Let
$$\mathcal R(w)=\{s\in S\mid ws<w\}\text{ and }\mathcal L(w)=\mathcal
R(w^{-1}).$$ The following result is well-known. See
\cite[2.4(i)]{KL79}.

\begin{Lemma}\label{LRset} If $w_1\leq_L w_2$, then $\mathcal
R(w_1)\supseteq \mathcal R(w_2)$. Hence, $w_1\sim_L w_2$ implies
$\mathcal R(w_1)=\mathcal R(w_2)$.
\end{Lemma}

Every left cell $\kappa$ defines a {\it left cell module}
$$E^\kappa:=\text{span}\{C_w\mid w\leq_L\kappa\}/\text{span}\{C_w\mid w<_L\kappa\},$$
where $w\leq_L\kappa$ means $w\leq_L y$ for some (equivalently, for all) $y\in\kappa$, and $w<_L y$ means $w\leq_L y$ but
$w\not\sim_Ly$.

It is known from \cite{KL79} that
cells for the symmetric group $\fS_r$, which is a Coxeter group with $S=\{(1,2),(2,3),\ldots,(r-1,r)\}$, are completely determined via
the Robinson--Schensted map and left cell modules form a complete set of  all irreducible
$\sH_{\mathbb Q(\up)}$-modules. We now give a brief description of these facts.


For non-negative integers $N,r$ with $N>0$, a composition $\lambda$
of $r$, denoted by $\lambda\models r$, is a sequence
$(\lambda_1,\lambda_2,\dots, \lambda_N)$ of non-negative integers
$\lambda_i$ with $N$ parts such that $|\lambda|=\sum_{i=1}^N
\lambda_i=r$. If such a sequence decreases weakly, then $\lambda$ is
called a partition of $r$ (with at most $N$ parts), denoted by
$\la\vdash r$.

The Young diagram $Y(\lambda)$ for a partition
$\lambda=(\lambda_1, \lambda_2, \ldots, \la_N)\vdash r$ is a collection of
boxes arranged in left-justified rows with $\lambda_i$ boxes in the
$i$-th row of $Y(\lambda)$.
Thus,
if $\la$ has $N$ parts, we say that $Y(\la)$ has $N$ rows.

 A $\lambda$-{\it tableau}  (or a {\it tableau of shape} $\la$) is obtained by
inserting integers  into boxes of $Y(\lambda)$. If the entries of a tableau $\ss$ are exactly $1,2,\ldots, r$, then $\ss$ is called an
{\it exact} tableau. The symmetric group $\mathfrak S_r$ acts on exact tableaux $\ss$ by permuting its
entries.

Let $\ts^\lambda$
(resp. $\ts_\lambda$) be the $\lambda$-tableau obtained from the
Young diagram $Y(\lambda)$ by inserting $1, 2, \cdots, r$ from left to
right (resp. top to bottom) along successive rows (resp. columns). For
example, for $\lambda=(4,3,1)$,
$$\ts^{\lambda}=\young(1234,567,8),\qquad
\ts_{\lambda}=\young(1468,257,3).$$
Let $R_i^\la$ be the $i$-th of $\ts^\lambda$, and let $\mathfrak S_\la$ be the row stabilizer subgroup
of $\fS_r$. For an exact tableau $\ss$, if $w(\ts^\lambda)=\ss$, write $w=d(\ss)$. Note that $d(\ss)$
is uniquely determined by $\ss$.

A $\lambda$-tableau $\ts$ is {\it row {\rm (resp.,} column$)$
increasing} if the entries in each row (resp. column) of $\ts$
strictly increase from left to right (resp., from top to bottom). It
is well known that there is a bijection between  the set of all row
increasing exact $\lambda$-tableaux and the set $\mathfrak
D_\lambda^{-1}$ of shortest left coset representatives of $\fS_\la$.
(Thus, $\mathfrak D_\nu$ is the set of right $\mathfrak S_\nu$-coset
representatives of minimal length.)  In particular, $\ss$ is a row
increasing exact $\lambda$-tableau  if and only if
 $d(\ss)\in \mathfrak D_\lambda^{-1}$.

For a partition $\lambda$,
an exact $\lambda$-tableau is {\it standard} if it is
both row increasing and column increasing.  Let $\bfT^s(\lambda)$ be the set of all standard
$\lambda$-tableaux.

Standard tableaux are used to describe elements of $\fS_r$ via the well-known \textit {Robinson-Schensted correspondence}.
 This map sets up a bijection
\begin{equation}\label{RSM}
\fS_r\longrightarrow\bigcup_{\la\in\Lambda^+(r)}\bfT^s(\la)\times\bfT^s(\la),\qquad w\RS(P(w),Q(w));
\end{equation}
   see, e.g., \cite[Cor.~8.9]{DDPW}. Here $P(w)$ is a standard tableau obtained by
   applying an insertion algorithm to $(j_1,j_2,\ldots,j_r)$, where $w(i)=j_i$,
   and $Q(w)$ is the recording tableau; see, e.g., \cite[(8.2.5)]{DDPW}. Moreover, we have $Q(w)=P(w^{-1})$.

One of the important applications of the Robinson-Schensted
correspondence is the decomposition of symmetric groups into
Kazhdan--Lusztig cells which are defined above Lemma~2.1. The
following result is given in \cite[Th.~1.4]{KL79}; see \cite[Th.~8.25]{DDPW}
for a purely combinatorial proof.

\begin{thm}\label{RS} Suppose $x, y\in \mathfrak S_r$. Then
\begin{enumerate} \item[(1)] $x\sim_{L} y$ if and only  if $Q(x)=Q(y)$.
\item[(2)] $x\sim_{R} y$ if and only if $P(x)=P(y)$.
\item[(3)] $x\sim_{LR} y$ if and only if  $P(x)$ and $P(y)$ have the same shape.\end{enumerate}
Moreover, let $\kappa_\la$ denote the left cell containing the
longest element $w_{0,\la}$ of $\fS_\la$ and
$S_\la:=(E^{\kappa_\la})^*$ the corresponding dual left cell module. Then
$\{S_{\la,\mathbb Q(\up)}\}_{\la\vdash r}$ is a complete set of
non-isomorphic irreducible right $\sH_{\mathbb Q(\up)}$-modules and, for any left cell $\kappa$, $S_{\la,\mathbb Q(\up)}\cong(E_{\mathbb Q(\up)}^{\kappa})^*$ if and only if $\kappa$ lie in the two-sided cell
containing $w_{0,\la}$.
\end{thm}
This result has a natural generalization to quantum Schur algebras; see \cite[(5.3.3)]{DR98}. 
We will develop a super version of this result in \S7.

\section{Super-representatives of double cosets}

 Let
$M(m+n,r)$ be the set of $(m+n)\times (m+n)$ matrices $A=(a_{i,j})$
with $a_{i,j}\in\mathbb N$ and $\sum a_{ij}=r$, and let
$M(m+n)=\cup_{r\geq0}M(m+n, r)$. Let
$${\ro(A)=(\sum_ja_{1,j},\sum_ja_{2,j},\ldots,\sum_ja_{n+m,j})}\atop
{\co(A)=(\sum_ja_{j,1},\sum_ja_{j,2},\ldots,\sum_ja_{j,n+m})}.$$
Define
\begin{equation}
\aligned
 M(m|n, r)&=\{(a_{ij})\in M(m+n, r)\colon a_{ij}\in\{0,1\} \text{ if
$\hat i+\hat j=1$}\},\\
 M(m|n)&=\bigcup_{r\geq0}M(m|n, r). \\
 \endaligned
\end{equation}

 Let $\Lambda(N, r)$ (resp. $\Lambda^+(N, r)$) be the
set of compositions (resp. partitions) of $r$ with $N$ parts. We also write
$\Lambda^+(r)$ for $\Lambda^+(r, r)$, the set of partitions of $r$,
and write $0$ for the unique element in $\Lambda(N,0)$.

For $(\lambda,\mu)\in\Lambda(m,r_1)\times
\Lambda(n,r_2)$, let
$$\lambda\vee\mu=(\lambda_1,\ldots,\la_m,\mu_1,\ldots,\mu_n)\in\Lambda(m+n,r_1+r_2).$$
Every element in $\Lambda(m+n,r)$ has the form $\la\vee\mu$ for some $(\lambda,\mu)\in\Lambda(m,r_1)\times
\Lambda(n,r_2)$ with $r_1+r_2=r$.
Let
\begin{equation}\label{La(m|nr)}
\aligned
\Lambda(m|n,r)&=\{\la|\mu:\la\in\Lambda(m,r_1),\mu\in\Lambda(n,r_2),\la\vee\mu\in\Lambda(m+n,r)\}\\
\Lambda^+(m|n,r)&=\{\la|\mu\in \Lambda(m|n,r):\la_1\geq\cdots\geq\la_m,\mu_1\geq\cdots\geq\mu_n\}.\endaligned
\end{equation}
Thus, we may identify $\Lambda(m|n,r)$ with
$\Lambda(m+n,r)$ via the map $\lambda|\mu\mapsto\la\vee\mu$.  Hence,
$$\fS_{\la|\mu}:=\fS_{\la\vee\mu}\cong
\fS_\la\times\fS_\mu$$
is well-defined. We will write $x|y\in\fS_{\la|\mu}$ to mean that
$x\in\fS_{\la^*}$ and $y\in\fS_{{}^*\!\mu}$, where
$$\la^*=\la\vee(1^{r-|\la|})\quad\text{ and }\quad{}^*\!\mu=(1^{r-|\mu|})\vee\mu.$$
In this notation, $\fS_{\la|\mu}=\fS_{\la^*}\fS_{{}^*\mu}.$ We will also write,
for any $A\in M(m|n,r)$,
$\ro(A)=\la|\mu$ or $\co(A)=\xi|\eta$ as elements in $\Lambda(m|n,r)$.


For notational simplicity, we will identify $\nu$ with the set
$\fS_\nu\cap S$. Let $\mathfrak D_\nu$ (resp. $\mathfrak D^+_\nu$ )
be the set of right $\mathfrak S_\nu$-coset representatives of
minimal (resp., maximal) length. Thus, for $\rho\models r$, the set
$$\mathfrak D_{\nu,\rho}=\mathfrak D_\nu\cap
\mathfrak D_\rho^{-1}\quad (\text{resp}., \mathfrak
D_{\nu,\rho}^+=\mathfrak D^+_\nu\cap (\mathfrak D_\rho^+)^{-1})
$$ consists of minimal (resp. maximal) double coset
representatives of double cosets in $\mathfrak S_\nu\backslash
\mathfrak S_r/\mathfrak S_\rho$. In particular, for
$\lambda|\mu\in\Lambda(m|n,r)$,  $\mathfrak D_{\lambda|\mu}$, $\mathfrak
D^+_{\lambda|\mu}$ and $\mathfrak D_{\lambda|\mu, \xi|\eta}$ are
defined.

For $\lambda|\mu\in\Lambda(m|n,r)$ and
$\xi|\eta\in\Lambda(m'|n',r)$, define
$$\fD_{\la|\mu}^{+,-}=\fD_{\la^*}^+\cap\fD_{{}^*\mu}\quad\text{ resp. }\quad
\fD_{\la|\mu}^{-,+}=\fD_{\la^*}\cap\fD_{{}^*\mu}^+,$$
and
$$\fD_{\la|\mu, \xi|\eta}^{+, -}=\fD_{\la|\mu}^{+, -}\cap (\fD_{\xi|\eta}^{+, -})^{-1}
\quad\text{resp. }\quad \fD_{\la|\mu, \xi|\eta}^{-,+}=\fD_{\la|\mu}^{-,+}\cap
 (\fD_{\xi|\eta}^{-,+})^{-1}.$$
It is clear that we have
\begin{equation}\label{double-sign}
\aligned
 \fD_{\la|\mu, \xi|\eta}^{+, -}&= \mathfrak
D_{\la^*,\xi^*}^+\cap \mathfrak D_{^*\mu,{}^*\eta}=\biggl\{x\in\fS_r\colon
{{sx<x,tx>x,\,\,\forall s\in\la,t\in\mu}\atop{xs<x,xt>x,\,\,\forall
s\in\xi,t\in\eta}}\biggr\},\\
\mathfrak  D_{\la|\mu, \xi|\eta}^{-,+}&= \mathfrak D_{\la^*,\xi^*}\cap
\mathfrak D_{^*\mu,{}^*\eta}^+=\biggl\{x\in\fS_r\colon
{{sx>x,tx<x,\,\,\forall s\in\la,t\in\mu}\atop{xs>x,xt<x,\,\,\forall
s\in\xi,t\in\eta}}\biggr\}.\\
\endaligned
\end{equation}
Moreover, if $\emptyset$ denotes the empty subset of $S$ associated with those
$\xi|\eta$ whose components are 0 or 1, then $\fD_{\la|\mu, \emptyset}^{+, -}=\fD_{\la|\mu}^{+, -}$.

The following result links the above sets with certain {\it trivial intersection property}.
\begin{Lemma}\label{TIP}  Let $\la|\mu,\xi|\eta\in\Lambda(m|n,r)$. For any
$d\in\mathfrak D_{\lambda|\mu, \xi|\eta}$, the following are equivalent.
\begin{enumerate}
\item[(1)]  $\mathfrak D_{\la|\mu,  \xi|\eta}^{+, -}\cap \fS_{\la|\mu}
  d \fS_{\xi|\eta}\neq \emptyset$;
\item[(2)] $\mathfrak D_{\la|\mu,  \xi|\eta}^{-,+}\cap \fS_{\la|\mu}
  d \fS_{\xi|\eta}\neq \emptyset$;
\item[(3)] $\fS_{\la^*} \cap d \fS_{^*\eta} d^{-1} =\{1\}$
and $\fS_{{}^*\!\mu}\cap d \fS_{\!\xi^*} d^{-1}=\{1\}$.
\end{enumerate}
Moreover, if one of the conditions holds, then
\begin{itemize}
\item[(1$'$)] $\mathfrak D_{\la|\mu, \xi|\eta}^{+,
-}\cap \fS_{\la|\mu} d \fS_{\xi|\eta}=\fD^+_{\la^*,\xi^*}\cap\fS_{\la^*} d\fS_{\xi^*}=\{d^*\}$;
 \item[(2$'$)] $\mathfrak
D_{\la|\mu, \xi|\eta}^{-,+}\cap \fS_{\la|\mu} d
\fS_{\xi|\eta}=\fD^+_{\mu^*,\eta^*}\cap\fS_{\mu^*} d\fS_{\eta^*}=\{{}^*\!d\}$.
\end{itemize}
\end{Lemma}
\begin{proof}Let $d^*$ be the unique element in
$\fD^+_{\la^*,\xi^*}\cap\fS_{\la^*} d\fS_{\xi^*}$. Thus,
 $\mathfrak D_{\la|\mu,  \xi|\eta}^{+, -}\cap \fS_{\la|\mu}
  d \fS_{\xi|\eta}\neq \emptyset$ is equivalent to the condition $d^*\in\fD_{^*\mu,{}^*\eta}$.
However,
$$\aligned
&\fS_{\la^*} \cap d \fS_{^*\eta} d^{-1} \not=\{1\}\quad(\text{resp.,
}\fS_{{}^*\!\mu}\cap
d \fS_{\xi^*\!} d^{-1}\not=\{1\})\\
\iff&\exists s\in\eta\text{ satisfying }t=dsd^{-1}\in\la\quad
(\text{resp., } \exists s\in\xi\text{ satisfying }t=dsd^{-1}\in\mu)\\
\iff&d^*s<d^*\quad(\text{resp., } td^*<d^*) \\
\iff&d^*\not\in\fD_{^*\mu,{}^*\eta}.\\
\endaligned$$
So (1) and (3) are equivalent. A similar argument shows that (2) is equivalent to the conditions $d^{-1}\fS_{\la^*}d\cap \fS_{^*\eta} =\{1\}$ and $ d^{-1}\fS_{{}^*\!\mu}d\cap \fS_{\!\xi^*}=\{1\}$. Hence,
(2) and (3) are equivalent. The last assertion follows from definition.
\end{proof}

It is well-known that double cosets of the symmetric group can be
described in terms of matrices. More precisely, 
there is a bijection
\begin{equation}\label{jmath}
\jmath:\fJ(N,r):=\{(\nu,w,\rho)\mid
\nu,\rho\in\Lambda(N,r),w\in\fD_{\nu,\rho}\}\lra M(N,r)
\end{equation} such
that if $\jmath(\nu,w,\rho)=A=(a_{i,j})$ then
$a_{i,j}=|R_i^\nu\cap wR_j^\rho|$, where $R_k^\la$ is the $k$-th row of $\ts^\la$. In other words, for
$\la=(\la_1,\la_2,\ldots,\la_N)$ and $1\le k\le N$,
$$R_k^\la=\{\la_1+\cdots+\la_{k-1}+1,\la_1+\cdots+\la_{k-1}+2,\ldots,
\la_1+\cdots+\la_{k-1}+\la_k\}.$$
Moreover, if $\jmath(\nu,w,\rho)=A$, then $\jmath(\rho,w^{-1},\nu)=A^t$, the transpose of $A$.
We now describe a ``super'' version of $\jmath$.
\begin{Prop}\label{bij}
Let
$$\fJ(m|n,r)=\bigcup_{\lambda|\mu, \xi|\eta\in
\Lambda(m|n, r)}\{(\la|\mu,d,\xi|\eta)\colon
d\in\mathfrak D_{\lambda|\mu, \xi|\eta},
  \mathfrak D_{\la|\mu,  \xi|\eta}^{+, -}\cap \fS_{\la|\mu}
  d \fS_{\xi|\eta}\neq \emptyset\}.$$ By restriction, the map $\jmath$ given in \eqref{jmath} induces a bijection
\begin{equation}\label{jmath--}
\jmath:\fJ(m|n,r)\longrightarrow M(m|n,r).
\end{equation}
\end{Prop}

\begin{proof} 
For $w,y\in \fS_r$,  it is well-known that
$$a_{ij}:=|R_i^{\lambda\vee\mu}\cap wR_j^{\xi\vee\eta}|=
|R_i^{\lambda\vee\mu}\cap yR_j^{\xi\vee\eta}|$$ whenever $\mathfrak
S_{\lambda|\mu}w\mathfrak S_{\xi|\eta}=\mathfrak
S_{\lambda|\mu}y\mathfrak S_{\xi|\eta}$.

For $\la|\mu\in\Lambda(m|n,r)$, if we put
 $R_i^\lambda=R_i^{\lambda\vee\mu}$ for $1\le i\le m$ and
$R_j^\mu=R_{m+j}^{\lambda\vee\mu}$ for $1\le j\le n$, then
$$a_{ij}=\begin{cases} | R_{i}^\lambda \cap
w R_j^\xi|, &\text{if $i\le m, j\le m$,}\\
| R_{i}^\lambda \cap w R_{j-m}^\eta|, &\text{if $i\le m, j\ge m+1$,}\\
| R_{i-m}^\mu \cap
w R_j^\xi|, &\text{if $i\ge m+1, j\le m$,}\\
| R_{i-m}^\mu  \cap w
R_{j-m}^\eta|, &\text{if $i\ge m+1, j\ge m+1$.}\\
\end{cases}
$$
Now,  $w\in\mathfrak D_{\lambda|\mu, \xi|\eta}^{+, -}$ if and only if both
$\fS_{\la^*} \cap w \fS_{^*\eta} w^{-1} =\{1\}$
and $\fS_{{}^*\!\mu}\cap w\fS_{\xi^*} w^{-1}=\{1\}$. This is equivalent to
$|R_{i}^\lambda \cap w R_{j-m}^\eta|\leq1$ and $| R_{i-m}^\mu \cap
w R_j^\xi|\leq1$ for all $1\le i\le m, m+1\le j\le m+n$, or $m+1\le i\le m+n,1\le j\le m$.
Hence, regarding $\fJ(m|n,r)$ as a subset of $\fJ(m+n,r)$, $\jmath$ sends $\fJ(m|n,r)$ into $M(m|n,r)$.
So the restriction is well-defined. The bijectivity follows from that of $\jmath$ and the argument above. (One may also use
Proposition \ref{alg} below to see the surjectivity.)
\end{proof}

 Let
$$\fJ(m|n,r)^{+,-}=\bigcup_{\lambda|\mu, \xi|\eta\in
\Lambda(m|n, r)}\{(\la|\mu,w,\xi|\eta)\colon w\in \mathfrak
D_{\lambda|\mu, \xi|\eta}^{+, -}\},$$ and define $\fJ({m|n,r})^{-,+}$
similarly. The following can be seen easily from Lemma \ref{TIP} and Proposition \ref{bij}.

\begin{Cor}\label{jmath+-} There are bijections
\begin{equation}\label{sjmath}
\jmath^{+,-}:\fJ(m|n,r)^{+,-}\lra M(m|n,r)\quad\text{\rm and}\quad
\jmath^{-,+}:\fJ({m|n,r})^{-,+}\lra M(m|n,r)
\end{equation}
 such that, if
$A=\jmath^{+,-}(\la|\mu,w,\xi|\eta)=\jmath^{-,+}(\la|\mu,w',\xi|\eta)$,
then $\ro(A)=\la|\mu$, $\co(A)=\xi|\eta$.
\end{Cor}

The map $\jmath^{+,-}$ will be used to introduce cell relations on $M(m|n,r)$ in \S7.

The map $\jmath$ can be used to explicitly describe the shortest and longest elements in the double coset
corresponding to a matrix $A\in M(N,r)$.
Write $w_A^-$ for $w$ if $\jmath(\nu,w,\rho)=A$. Then $w_A^-$ is the shortest element
 in the double coset $\fS_\nu w\fS_\rho$, Let $w_A^+$ be the longest element
in $\fS_\nu w_A^-\fS_\rho$.  By \cite{Du96} (or \cite[Exer.
8.2]{DDPW}, $w_A^-$ (resp. $w_A^+$) can be computed as follows:
construct a pseudo-matrix $A_-$ associated with $A$ by replacing
$a_{1,1}$ by the sequence consisting of the first $a_{1,1}$ integers
of $\{1,2,\ldots,r\}$, $a_{1,2}$ by the sequence of the next
$a_{1,2}$ integers, etc., from left to right down successive rows,
and then form the permutation $w_A^-$ which is obtained by reading
$A_-$ from left to right inside the sequences and from top to
bottom, and followed by left to right along successive columns.
\begin{example}\label{AAA}
$\text{ If } \,\,A=\begin{pmatrix} 2 & 0 & 1\\
1 & 2 & 0\\
1 & 2 & 1\\
\end{pmatrix},\,\,\text{ then }\,\, A_-=\begin{pmatrix}  (1,2) &  \emptyset & 3\\
                                  4 & (5,6)& \emptyset\\
                                  7 & (8, 9) & 10\\
                                  \end{pmatrix}
$ and $$w_A^-=(1,2,4,7,5, 6, 8, 9, 3,10).$$
\end{example}
By reversing the integers in each row of $A_-$ and form a
pseudo-matrix $A_+$, the permutation $w_A^+$ is obtained by reading
$A_+$ from left to right inside the sequences and from bottom to
top, and followed by left to right along successive columns.
For the example above, we have
$$A_+=\begin{pmatrix}  (3,2) &  \emptyset & 1\\
                                  6 & (5,4)& \emptyset\\
                                  10 & (9, 8) & 7\\
                                  \end{pmatrix}\,\,\,\text{ and
                                  }\,\,\,
                                  w_A^+=(10,6,3,2,9,8,5,4,7,1)
                                  $$

We now generalize this construction to the elements in
$\fJ(m|n,r)^{+,-}$ and $\fJ({m|n,r})^{-,+}$. Write $w_A^{+,-}$ for
$w$ if $\jmath^{+,-}(\nu,w,\rho)=A$ and $w_A^{-,+}$ for $w'$ if
$\jmath^{-,+}(\nu,w',\rho)=A$. Suppose $A=(a_{ij})\in M(m|n, r)$. By
regarding $A$ as an element in $M(m+n, r)$, construct a
pseudo-matrix $A_{+,-}$ (resp., $A_{-,+}$) by reversing the integers
in each row of $A_-$ for the first $m$ rows (resp. last $n$ rows).

Now, define the permutation $w_A^{+,-}$ (resp., $w_A^{-,+}$) by
reading $A_{+,-}$ (resp., $A_{-,+}$) from left to right inside the
sequences and from bottom to top (resp., top to bottom), followed by
left to right along the first $m$ successive columns, and then from
top to bottom (resp., bottom to top) for the next $n$ successive columns.

\begin{example} If $A$ is the matrix as given in Example \ref{AAA},
then $A\in M(m|n,10)$ for $m=1$ and $n=2$, $\ro(A)=\lambda|
\mu=(3)|(3,4)$, $\co(A)=\xi| \eta=(4)|(4,2)$, and
$$A_{+,-}=\begin{pmatrix}  (3,2) &  \emptyset & 1\\
                                  4 & (5,6)& \emptyset\\
                                  7 & (8,9) & 10\\
                                  \end{pmatrix},\quad\,\,
                                  \,\,
A_{-,+}=\begin{pmatrix}  (1,2) &  \emptyset & 3\\
                                  6 & (5,4)& \emptyset\\
                                  10 & (9, 8) & 7\\
                                  \end{pmatrix}$$
Hence, $w_A^{+,-}=(7,4,3,2,5,6,8,9,1,10)$ and
$w_A^{-,+}=(1,2,6,10,9, 8, 5, 4, 7,3)$.
\end{example}

\begin{Prop} \label{alg} Maintain the notation introduced above. If
 $A\in M(m|n, r)$ with
$co(A)=\xi|\eta$ and $ro(A)=\lambda|\mu$ then $w_A^{+,-}\in
\mathfrak D_{\lambda|\mu, \xi|\eta}^{+, -}$ and $w_A^{-,+}\in
\mathfrak D_{\lambda|\mu, \xi|\eta}^{-,+}$.
\end{Prop}

\begin{proof}  Let $w=(i_1, i_2,
\cdots, i_r)\in\fS_r$ where $w(k)=i_k$. Let $s_k$ be the basic transposition
which switches $k$ and $k+1$.  Then $ w s_k$ is obtained from $w$ by
switching $i_k$ and $i_{k+1}$. On the other hand, $s_k w $ is
obtained from $w$ by switching $i_p$ and $i_q$ if $\{i_p, i_q\}=\{k,
k+1\}$. Since
\begin{equation}\label{length} l(w)=\sum_{j=1}^r \# \{(j,k)\mid j<k,
i_j>i_k\}.\end{equation} it follows that $l(ws_k )=l(w)+1$ if and
only if $i_k<i_{k+1}$, while $l(s_k w)=l(w)+1$ if and only if $\{k,
k+1\}$ is a subsequence of $\{i_1, i_2, \cdots, i_r\}$. Now, the
result follows immediately by taking $w=w_A^{+,-}$ of $w_A^{-,+}$.
\end{proof}

\section{Young supertableaux and RSK super-correspondence}

Before generalizing \ref{RS}, we need some combinatorial preparations.

For $\lambda\in \Lambda^+(n, r)$ and $\mu\in \Lambda(n, r)$, a
$\lambda$-tableau $\ssS$ of {\it content} (or {\it type}) $\mu\models n$ is
the tableau obtained from $Y(\lambda)$ by inserting each box with
numbers $i, 1\le i\le n$, such that the number $i$ occurring in
$\ssS$ is $\mu_i$. If the entries in $\ssS$ are weakly increasing in
each row (resp., column) and strictly increasing in each column
(resp. row), $\ssS$ is called a {\it row (resp., column)
semi-standard} $\lambda$-tableau of content $\mu$. A row semistandard tableau
is simply called semistandard tableau sometimes.  Let $\bfT(\la,\mu)$ (resp.,
$\bfT^{ss}(\lambda, \mu)$) be the set of all
$\lambda$-tableaux (resp., semi-standard $\la$-tableau) of content $\mu$.
If $\bfT^{ss}(\lambda, \mu)\neq\emptyset$, then $\lambda\unrhd\mu$.

Fix two non-negative integers $m, n$ with $m+n>0$, define
\begin{equation}\label{cellweight}
 \Lambda^+(r)_{m|n}=\{\lambda\in \Lambda^+(r), \lambda_{m+1}\le
n\}
\end{equation}
If $\la\in \Lambda^+(r)_{m|n}$, then $Y(\la)$ is inside a hook of height $m$ and base $n$ and is called a
$(m,n)$-hook Young diagram. See,e.g.,\cite[2.3]{BR} where $\Lambda^+(r)_{m|n}$ is denoted as $H(m,n;r)$.

The set $\Lambda^+(r)_{m|n}$ is in general not a subset of $\Lambda(m|n, r)$ or $\Lambda^+(m|n, r)$;
see \eqref{La(m|nr)}. However, each partition
$\lambda\in \Lambda^+(r)_{m|n}$ uniquely determines a pair of
partitions $\la'$ and $\la''$ with
\begin{equation}\label{pair}
\la'=(\la_1,\ldots, \la_m),\qquad\la''=(\la_{m+1},\la_{m+2},\ldots)^t,
\end{equation}
where, for $\nu\vdash r$, $\nu^t$ denotes the partition dual to
$\nu$. (In other words, the Young diagram $Y(\nu^t)$ is the
transpose of $Y(\nu)$.)  The condition $\la_{m+1}\leq n$ implies
$\la'|\la''\in \Lambda^+(m|n, r)$. Thus, we obtain an injective map
\begin{equation}\label{map wp}
\Lambda^+(r)_{m|n}\longrightarrow\Lambda^+(m|n,r),\quad\la\longmapsto(\la',\la'').
\end{equation}
 The pair $(\la',\la'')$ is sometimes called a dominant weight in the
representation theory of quantum general linear superalgebra
$U_q(\mathfrak{gl}(m|n))$; see, e.g., \cite{Moon} and \cite{Mit}.
Note that, for $\la\in \Lambda^+(r)_{m|n}$, $Y(\la)$ is called an   $(m, n)$-hook diagram
in \cite[\S4.1]{BKK}.

We now introduce, following \cite[\S1.2]{Ser} (cf. \cite[Def.~4.1]{BKK}), the notion of
semistandard $\lambda$-supertableau of content $\mu|\nu$.

Let
$\lambda\in \Lambda^+(r)_{m|n}$, $\mu|\nu \in \Lambda(m|n, r)$. A
$\lambda$-tableau $\ssS$ of content $\mu\vee\nu$ is called a \textsf
{semi-standard $\lambda$-supertableau of content $\mu|\nu$} if
\begin{enumerate}
\item the entries in $\ssS$ are weakly increasing in each row
and each column of $\ssS$;
\item the numbers in
$\{1, 2, \cdots, m\}$ are strictly increasing in the columns and the numbers in
$\{m+1, m+2, \cdots, m+n\}$ are strictly increasing in the rows.
\end{enumerate}
In other words, a semi-standard $\lambda$-supertableau of content $\mu|\nu$ is a tableau of content
$\mu|\nu$ such that the tableau $\ssT|_{[1,m]}$ obtained by removing entries $m+1,\ldots,m+n$ is a (row)
semi-standard tableau of content $\mu$ and the tableau obtained from $\ssT$ by removing $\ssT|_{[1,m]}$
is a column semi-standard {\it skew}-tableau of content $\nu$.

Let $\bfT^{sss}(\lambda, \mu|\nu)$ be the set of all semi-standard
 $\lambda$-supertableaux of content $ \mu|\nu$. Clearly,
$\bfT^{sss}(\lambda, \mu|0)=\bfT^{ss}(\lambda, \mu)$. Moreover, for $\ssS\in
\bfT^{sss}(\lambda, \mu|\nu)$, the subtableau obtained by removing all
$i$-th rows from $\ssS$ with $1\le i\le m$ is column semistandard.

\begin{example}\label{ssT_la} For any $\lambda\in \Lambda^+(r)_{m|n}$, there is a unique $\la$-tableau $\ssT_\la$ of content
$\la'|\la''$. For example, if $\la=(4,4,3,2,2,1)$ and $m=2,n=4$, then $\la'=(4,4)$, $\la''=(4,3,1)$ and
$$\ssT_\la=\young(1111,2222,345,34,34,3)$$
\end{example}

  The following result is known; see \cite[Theorem~2]{Ser} or \cite[Lemma 4.2]{BKK}. For completeness, we include a proof.

\begin{Lemma}\label{equi} For a partition $\lambda\in \Lambda^+(r)$, $\bfT^{sss}(\lambda, \mu|\nu)\neq \emptyset$ for some $\mu|\nu\in \Lambda(m|n, r)$ if and only if $\lambda\in \Lambda^+(r)_{m|n}$.
\end{Lemma}
\begin{proof} If $\lambda\in \Lambda^+(r)_{m|n}$, then $\bfT^{sss}(\la,\la'|\la'')=\{\ssT_\la\}\neq\emptyset$. Conversely, suppose $\ssT\in \bfT^{sss}(\lambda, \mu|\nu)$. Then the numbers $1,2,\ldots,m$ do not appear in the rows below row m. Let  $\ssT''$ be the transpose of the tableau obtained by
removing the first $m$ rows from $\ssT$.  Replacing every entry $x$ in $\ssT''$ by $x-m$ yields a semistandard $\la''$-tableau with content $\nu^{(2)}$ for some $\nu^{(2)}\in\Lambda(n,r_2)$. Now, $\la''\trianglerighteq \nu^{(2)}$ implies that $\la_{m+1}$, which is the number of parts of $\la''$, is less than or equal to the number of parts of $\nu^{(2)}$, which is $\leq n.$
Hence, $\la\in\Lambda^+(r)$.
\end{proof}

The RSK super-correspondence is about a bijection between $M(m|n,r)$ and the pairs of
semistandard super tableaux of the same shape. Since the correspondence will be used to describe super-cells and associated modules, our construction relies on the relationship between Kazhdan-Lusztig cells of $\fS_r$ and their combinatorial characterization.

For a fixed $w\in \mathfrak S_r$ and $T\in \bfT(\lambda,\mu)$. Define
$w_T\in \mathfrak S_r$ by letting $w_T (\ts^\mu) $ be the
row standard $\mu$-tableau such that the integers in the $i$th row of
$w_T(\ts^\mu) $ are the entries of $w(\ts^\la) $
whose positions are the same as those of the $\mu_i$ entries $i$ in $T$.
 It is easy to see that the map
$\bfT(\lambda,\mu)\to\mathfrak D_\mu^{-1},\,\,T\mapsto w_T$ is
bijective. The inverse $T_w^{\la,\mu}$ of this map can be defined as follows: for
$x\in \mathfrak D_{\mu}^{-1}$, define $T_w^{\la,\mu} (x)\in
\bfT(\lambda,\mu)$ by specifying that, for all $i,j$, if the entry
in $(i, j)$ position of $w(\ts^\la)$ is $a$, then the entry in
the same position in $T_w^{\la,\mu} (x)$ is the row index of $a$ in
the row standard $\mu$-tableau $x(\ts^\mu)$.

\begin{example} If $\la=(431)$, $\mu=(3,2,2,1)$, $w=w_{0, \lambda}$ the
longest element in  $\mathfrak S_\lambda$, and
$T=\young(1112,234,3)$, then
$$\qquad w_{0, \la}(\ts^{\la}) =\young(4321,765,8)\text{
and }(w_{0, \la})_T(\ts^{\mu})=\young(234,17,68,5)$$ If we
write $\mu$ as $\xi|\eta=(3,2)|(2,1)$, then
$$({w_{0, \la}})_T w_{0, \xi}(\ts^{\xi\vee\eta})=\young(432,71,68,5)$$
\end{example}
Observe from the example that the tableaux  $({w_{0, \la}})_T w_{0,
\xi}(\ts^{\xi\vee\eta})$, where $\mu=\xi|\eta\in\Lambda(m|n,r)$,  is obtained from $({w_{0, \la}})_T
(\ts^{\xi\vee\eta})$ by reversing  the entries in the
$i$th-rows for each $i,1\le i\le m$.

We say that  $\mathbf i=(i_1, i_2, \cdots, i_k)$ is a subsequence of
$\mathbf j=(j_1, j_2, \cdots, j_l)$ if it is obtained from $\mathbf
j$ by deleting some entries of $\mathbf j$.

\begin{Lemma}\label{sss} For $\la\vdash r, \mu|\nu\in\Lambda(m|n,r)$, and $\ssT\in \bfT(\la, \mu\vee\nu)$, let $(w_{0,\la})_{\ssT}w_{0, \mu} =(j_1,j_2, \cdots, j_r)$ be the permutation sending $i$ to $j_i$. Then,
$\ssT\in \bfT^{sss}(\la, \mu|\nu)$ if and only if the rows $R_i$ and
columns $C_j$ of $w_{0, \la}(\ts^\la) $ are all subsequence of
$j_1, j_2, \cdots, j_r$.
\end{Lemma}

\begin{proof}  If $y=(w_{0,\la})_{\ssT}w_{0,
\mu} =(j_1,j_2, \cdots, j_r)$, then $(w_{0,\la})_{\ssT}w_{0,
\mu}(\ts^{\mu\vee\nu})$ has sequence $j_1,\ldots, j_{\mu_1}$
in the first row, $j_{\mu_1+1},\ldots,j_{\mu_1+\mu_2}$ in the second
and so on, and the first $m$ rows are obtained by reversing the
first $m$ rows of $(w_{0,\la})_{\ssT}(\ts^{\mu\vee\nu})$,
while the next $n$ rows are the same as the corresponding rows of
$(w_{0,\la})_{\ssT}(\ts^{\mu\vee\nu})$. In particular, the
first $m$ rows are decreasing, while the next $n$ rows are
increasing.

A column of $\ssT$ has the form $a_1a_2\ldots a_l a..abb\ldots$
(from top to bottom) with $p$ $a$'s, $q$ $b$'s and so on for some $l,p,q,\ldots\geq0$, where
$a_1<a_2\cdots<a_l\leq m<a<b\cdots\leq m+n$. By definition, the first $l$ members of
$C_j$ are placed in different rows of
$(w_{0,\la})_{\ssT}(\ts^{\mu\vee\nu})$ (and hence of
$(w_{0,\la})_{\ssT}w_{0, \mu}(\ts^{\mu\vee\nu})$) with row indexes
$a_1,a_2,\ldots,a_l$, and then the next $p$ members of $C_j$ are
placed (as a whole) in row $a$ and the next $q$ members are in row
$b$, and so on. Note that $a, b, \cdots, $ are strictly great than $m$.
Hence, $C_j$ is a subsequence of $j_1,j_2, \cdots, j_r$. This proves
the result for  $C_j$ for all $j'$s.

Likewise, the $i$th row of $\ssT$ has the form $a.. abb\ldots a_1\ldots a_l$ with $p$ $a$'s, $q$ $b$'s and
so on for some $l,p,q,..\geq0$, where $a<b<\cdots\leq m<a_1<\cdots<a_l\leq m+n$.
Thus, the first $p$ members of $R_i$ (as a whole) form part of
the row $a$ of $(w_{0,\la})_{\ssT}(\ts^{\mu\vee\nu})$ (and hence of
$(w_{0,\la})_{\ssT}w_{0, \mu}(\ts^{\mu\vee\nu})$ since they are decreasing), the next $q$ members form part of row $b$, and so on. Then the members of $R_i$ are placed in
different rows between row $m+1$ and row $m+n$. Hence, $R_i$ is a
subsequences of $j_1,j_2, \cdots, j_r$. The argument above also
shows that if either the $i$th row or $j$th column of $\ssT$ is not
(weakly) increasing, then either $R_i$ or $C_j$ is not a subsequences of
$j_1,j_2, \cdots, j_r$, proving the lemma.
\end{proof}

The following result is the key to the establishment of the RSK super-correspondence.
\begin{Prop}\label{supersemi}Suppose $\mu|\nu\in \Lambda(m|n, r)$ and
$\la\in \Lambda(r)^+$. If $\varpi_\la$ denotes the right cell of
$\mathfrak S_r$ containing $w_{0, \la}$, then
$$
\mathfrak D_{\!\la|0, \mu|\nu}^{+,-}\cap \varpi_{\la} =\{(w_{0,
\la})_{\ssT} w_{0, \mu} \mid \ssT\in \bfT^{sss} (\la, \mu|\nu)\}.$$
\end{Prop}

\begin{proof}By \cite[3.2]{Du95},\footnote{A right action was used for the
symmetric group $\mathfrak S_r$ in \cite{Du95}. Thus, the left cell containing $w_{0,\la}$
was used there.} or more precisely, \cite[Lem. 8.20]{DDPW}, we have
$$\varpi_\la=\{(w_{0, \la})_{\ts}\mid\ts\in \bfT^s(\la)\},$$
where $(w_{0, \la})_{\ts}$ is simply defined by $(w_{0,
\la})_{\ts}(\ts)=w_{0, \la}(\ts^\la)$. Hence,
$$\mathfrak D_{\!\la|0,
\mu|\nu}^{+,-}\cap \varpi_{\la} =\{(w_{0, \la})_{\ts}\mid
\ts\in\bfT^s(\la),(w_{0, \la})_{\ts}\in \fD_{\!\la|0,\mu|\nu}^{+,-}\}.$$
We now prove that
$$\{(w_{0, \la})_{\ts}\mid
\ts\in\bfT^s(\la),(w_{0, \la})_{\ts}\in \fD_{\!\la|0,\mu|\nu}^{+,-}\}=
\{(w_{0, \la})_{\ssT} w_{0, \mu} \mid \ssT\in \bfT^{sss} (\la,
\mu|\nu)\}.$$

If we put $\omega=(1^r)$, then $\bfT^s(\la)=\bfT^{ss}(\la,\omega)$.
Suppose $\ts\in\bfT^{s}(\la)$ and $(w_{0, \la})_{\ts}\in
\fD_{\!\la|0,\mu|\nu}^{+,-}$. Then by definition $x=(w_{0,
\la})_{\ts}w_{0,\mu}\in\fD_{\mu\vee\nu}^{-1}$. Let
$\ssT=T_{w_{0,\la}}^{\la,\mu\vee\nu}(x)\in\bfT(\la,\mu|\nu)$ so that
$x=(w_{0,\la})_\ssT$ and $(w_{0, \la})_{\ts}=(w_{0,\la})_\ssT
w_{0,\mu}$. We claim that $\ssT\in\bfT^{sss}(\la,\mu|\nu)$. Indeed,
suppose $(w_{0,\la})_{\ts}=(i_1, i_2, \ldots, i_r)$.
By applying Lemma \ref{sss} to the case where $\mu|\nu=\omega|0$,   $\ts$ is
standard implies that the rows $R_i$ and columns $C_j$ of
$w_{0,\la}(\ts^\la)$ are subsequences of $i_1, i_2, \ldots, i_r$.
Thus, the same lemma (applied to $(w_{0,\la})_\ssT w_{0,\mu}=(i_1,
i_2, \ldots, i_r)$) implies that $\ssT\in \bfT^{sss}(\la,\mu|\nu)$.

Conversely, for any $\ssT\in \bfT^{sss}(\la, \mu|\nu)$, assume
$(w_{0,\la})_{\ssT}w_{0, \mu} =(j_1,j_2, \ldots, j_r)$. By Lemma
\ref{sss}, the rows $R_i$ and columns $C_j$ of $w_{0,\la}(\ts^\la)$
are subsequences of $j_1,j_2, \ldots, j_r$. Suppose
$R_1=\{j_{i_1},\ldots,j_{i_{\la_1}}\}$,
$R_2=\{j_{i_{\la_1+1}},\ldots\}$ and so on. Then the $\la$-tableau
$\ts$ obtained by putting $i_1,\ldots,i_{\la_1},i_{\la_1+1},\ldots$
from left to right down successive rows is standard and $(w_{0,
\la})_{\ts}=(w_{0,\la})_\ssT w_{0,\mu}$.\end{proof}

\begin{Cor}\label{multiplicity}
For $\mu|\nu\in\Lambda(m|n,r)$,  $\fD_{\emptyset,\mu|\nu}^{+,-}$ is a
union of left cells. For $\la\vdash r$, if $K_\la$ denotes the two-sided cell
containing $w_{0,\la}$, then the number
$m_{\la,\mu|\nu}$ of left cells in $K_\la\cap
\fD_{\emptyset,\mu|\nu}^{+,-}$ is $|\bfT^{sss}(\la,\mu|\nu)|$.
\end{Cor}

\begin{proof}Since
$$\fD_{\emptyset,\mu|\nu}^{+,-}=\{w\in\fS_r\mid
\sR(w)\supseteq\mu, \sR(w)\cap\nu=\emptyset\}=\{w\in
(\fD_\mu^+)^{-1}\mid \sR(w)\cap\nu=\emptyset\},$$ and
$(\fD_\mu^+)^{-1}$ is a union of left cells $\kappa$ satisfying
$\sR(\kappa)\supseteq\mu$, it follows that
$\fD_{\emptyset,\mu|\nu}^{+,-}$ is a union of left cells $\kappa$ in
$(\fD_\mu^+)^{-1}$ satisfying $\sR(\kappa)\cap\nu=\emptyset$. Hence,
by Proposition~\ref{supersemi},
$$m_{\la,\mu|\nu}=|
\fD_{\emptyset,\mu|\nu}^{+,-}\cap K_\la\cap\varpi_\la|=|\fD_{\la,\mu|\nu}^{+,-}\cap\varpi_\la|
=|\bfT^{sss}(\la,\mu|\nu)|,$$ as required.
\end{proof}

Assume $\mu|\nu\in\Lambda(m|n,r)$. For $T\in\bfT^{ss}(\la,\mu^*)$,
replacing $\nu_1$ entries $m+1,\ldots,m+\nu_1$ of $T$ by $m+1$,
$\nu_2$ entries $m+\nu_1+1,\ldots,m+\nu_1+1+\nu_2$ by $m+2$, and so
on, yields a $\la$-tableau $T^\diamond$ of type $\mu\vee\nu$, which
may not be  in $\bfT^{sss}(\la,\mu|\nu)$. Let
$$\bfT^{ss}(\la,\mu^*)^\diamond=\{T\in \bfT^{ss}(\la,\mu^*):
T^\diamond\in\bfT^{sss}(\la,\mu|\nu)\}.$$ Thus, we may identify
$\bfT^{sss}(\la,\mu|\nu)$ as the subset
$\bfT^{ss}(\la,\mu^*)^\diamond$ of $\bfT^{ss}(\la,\mu^*)$. This
identification is compatible with the inclusion
$\fD_{\la,\mu|\nu}^{+,-}\cap\varpi_\la\subseteq\fD_{\la,\mu^*}^+\cap\varpi_\la$.

We are now ready to describe RKS super-correspondence.

Suppose $w\in \mathfrak D_{\lambda|\mu, \xi|\eta}^{+,-}$. Let
$(P(w), Q(w))=(\ss, \ts)$ be the image of $w$ under the
Robinson-Schensted map, i.~e., $w\RS(\ss, \ts)$. Let $\nu^t$ be the shape of $\ss$ where
$\nu^t$ is the partition dual to $\nu$.
Define $x, y\in \mathfrak
S_r$ such that $P(x^{-1})=\ss$, $Q(x^{-1})=\ts_{\nu^t}$,
$P(y)=\ts_{\nu^t}$ and $Q(y)=\ts$.
 Since $P(w_{0,\nu})=Q(w_{0,\nu})=\ts_{\nu^t}$, by Theorem~\ref{RS}, $$w_{0,\nu}\sim_L x^{-1}\sim_R w\quad \text{ and }\quad  w_{0,\nu}\sim_R
y\sim_L w.$$
Thus, by Lemma \ref{LRset}, $\sR(x)=\sL(w)$, $\sR(y)=\sR(w)$ and $\sL(x)=\sL(y)=\nu$. This implies that
$x\in \mathfrak D_{\!\nu|0, \lambda|\mu}^{+, -}\cap
\varpi_\nu$ and $y\in \mathfrak D_{\!\nu|0, \xi|\eta}^{+, -}\cap
\varpi_\nu$, where $\varpi_\nu$ is the right cell of $\mathfrak S_r$
which contains $w_{0, \nu}$. By Proposition~\ref{supersemi}, there
is a pair of semi-standard $\nu$-tableaux $(\ssS_w, \ssT_w)\in
\bfT^{sss}(\nu, \lambda|\mu)\times \bfT^{sss} (\nu, \xi|\eta)$,
 which are determined uniquely by $x$ and $y$, respectively.
In particular, $\nu\in \Lambda^+(r)_{m|n}$. Thus, we obtain a
map
\begin{equation}\label{genrs}
\partial=\partial_{\lambda|\mu, \xi|\eta}^{+,-}:\mathfrak D_{\lambda|\mu, \xi|\eta}^{+, -}\longrightarrow
\!\!\bigcup_{\nu\in \Lambda^+(r)_{m|n}
 }  \bfT^{sss}(\nu, \lambda|\mu)\times \bfT^{sss}(\nu,
\xi|\eta),\,\, w\longmapsto (\ssS_w, \ssT_w).
\end{equation}
The symmetry of the Robinson-Schensted correspondence implies that the map $\partial$ satisfies a similar property:
$$\partial(w)=(\ssS_w, \ssT_w)\implies\partial(w^{-1})=(\ssT_w,\ssS_w).$$

\begin{thm}\label{RSKs} The maps $\partial_{\lambda|\mu, \xi|\eta}^{+,-}$,
 for any $\lambda|\mu, \xi|\eta\in \Lambda(m|n, r)$, are bijection which induce a bijective correspondence
$$M(m|n,r)\RSKs\bigcup_{\lambda|\mu, \xi|\eta\in \Lambda(m|n, r)\atop\nu\in \Lambda^+(r)_{m|n}
 }  \bfT^{sss}(\nu, \lambda|\mu)\times \bfT^{sss}(\nu,
\xi|\eta),\,\, A\RSKs(\ssS(A), \ssT(A)).$$
Moreover, if $A\RSKs(\ssS, \ssT)$ then $A^t\RSKs(\ssS, \ssT)$.
\end{thm}
\begin{proof} By Proposition~\ref{bij} and (\ref{genrs}),  we need
only construct the inverse map $\partial^{-1}$ of
$\partial=\partial_{\lambda|\mu, \xi|\eta}^{+,-}$ for the
first assertion.
 By Proposition~\ref{supersemi}, each pair
$(\ssS,\ssT)\in \bfT^{sss}(\nu, \lambda|\mu)\times \bfT^{sss}(\nu,
\xi|\eta)$ defines two elements $x=(w_{0,\nu})_\ssS
w_{0,\la}\in\fD_{\!\nu|0, \lambda|\mu}^{+, -}\cap\varpi_\nu$ and
$y=(w_{0,\nu})_\ssT w_{0,\xi}\in\fD_{\!\nu|0, \xi|\eta}^{+,
-}\cap\varpi_\nu$. By \cite[3.2]{Du95} (cf. footnote 2),
$P(x)=P(y)=\ts_{\nu^t}$. If $w\RS(Q(x),Q(y))$, then
$x^{-1}\sim_Rw\sim_Ly$. Thus, $\sL(w)=\sR(x)$ and $\sR(w)=\sR(y)$.
Hence, $w\in \mathfrak D_{\lambda|\mu, \xi|\eta}^{+, -}$, and
$\partial^{-1}(\ssS,\ssT)=w$. The last assertion is clear.
\end{proof}

We give an example to illustrate the proof.
\begin{example}  Let $\nu=(3,3,1)$ and $m=1$ and $n=3$. Then
$\nu'=(3)$,  $\nu''=(2,1,1)$ and
$\bfT^{sss}(\nu,\nu'|\nu'')=\{\ssT\}$ with
$$\ssT=\young(111,234,2).$$ Thus, $$ w_{0,\nu}(\mathsf
t^\nu)=\young(321,654,7),\text{ and } (w_{0,\nu})_\ssT(\mathsf
t^{\nu'\vee\nu''})=\young(123,67,5,4).$$ Hence,
$x=y=(w_{0,\nu})_\ssT
w_{0,\nu'}=(3,2,1,6,7,5,4)\in\fD_{\nu|0,\nu'|\nu''}^{+,-}$ and
$Q(x)=Q(y)=\young(145,26,37)$. Hence, $\partial^{-1}(\ssT,\ssT)=w\in\fD_{\nu'|\nu'',\nu'|\nu''}^{+,-}$ where $w\RS (Q(x),Q(y))$.
\end{example}

This bijective correspondence is called the
\textit{Robinson--Schensted--Knuth super-correspondence.}
We will write, for any $A\in M(m|n,r)$,
$A\RSKs(\ssS,\ssT)$ if $\ssS(A)=\ssT$ and $\ssT(A)=\ssT$.
\begin{rem} (1) This
correspondence is the super version of the correspondence given in
\cite[\S5.3]{DR98}; cf. \cite[Remark~9.26]{DDPW}. This correspondence is different from the so-called
$(m,n)$-RoSch correspondence described in \cite[2.5]{BR}.
\end{rem}


\section{Signed $q$-permutation modules and Quantum Schur superalgebras}

  The
Hecke algebra  $\mathcal H=\mathcal H(r)$ associated to the
symmetric group $\mathfrak S_r$ is an associative $\ZC$-algebra
generated by $T_i, 1\le i\le r-1$ subject to the relations (where
$\bsq=\up^2$)
\begin{equation}\label{heckea}\begin{cases}
 T_i^2=(\bsq-1)T_i+\bsq, & \text{for $1\le i\le
r-1$,}\\
T_i  T_j=  T_j  T_i, &\text{for $1\le i<j\le r-1$,} \\ T_i  T_{i+1}
 T_i= T_{i+1}  T_i
 T_{i+1}, &\text{for  $1\le i\le r-2$.}\\
\end{cases}
\end{equation}
For any commutative ring $R$ which is a $\ZC$-algebra, let $\mathcal
H_R$ be the algebra obtained by base change to $R$. Let $v,q$ be the
images of $\up,\bsq$ in $R$, respectively.

For each $\lambda|\mu\in \Lambda(m|n, r)$, define
$$x_\lambda=\sum_{w\in\mathfrak S_{\la^*}} T_w, \quad
y_\mu=\sum_{w\in \mathfrak S_{{}^*\!\mu}} (-q)^{-l(w)}T_w$$ where
$l(w)$ is the length of $w$. The $\sH$-module $x_\la\sH$ is called
a $q$-permutation module. We call $x_\la y_\mu\sH$ a {\it signed
$q$-permutation module}. These modules share certain nice properties with
$q$-permutation modules; cf. e.g., \cite[\S7.6]{DDPW}. We continue to follow the notation used in \S3.
Thus, a composition $\la\models r$ is identified with the set $\fS_\la\cap S$.

\begin{Lemma}\label{ABC}
Let $\lambda|\mu\in \Lambda(m|n, r)$.
\begin{enumerate}
\item[(1)] The right $\sH_R$-module
$x_\la y_\mu\sH_R$ is free with basis $\{x_\la y_\mu
T_d\}_{d\in\fD_{\la|\mu}}$.

\item[(2)] $x_\la y_\mu\sH_R=\{h\in\sH\colon
T_{s} h=q h,  T_th=- h,\,\forall s\in \lambda^*, t\in {}^*\!\mu\}.$

\item[(3)] $(\sH_Rx_\la y_\mu)^*:=\text{\rm Hom}_R(\sH_Rx_\la y_\mu,R)\cong x_\la
y_\mu\sH_R$.
\end{enumerate}
\end{Lemma}

\begin{proof} Statement (1) is clear.
For $h=\sum_wf_wT_w\in\sH_R$, $T_{s} h=q h,  T_th=- h$ imply
$f_w=f_{sw}$ and $f_{tw}=-q^{-1}f_{w}$ for all $s\in\la$, $t\in \mu$
with $tw>w$, which force
$$h=\sum_{x\in\fS_{\la^*},y\in\fS_{^*\!\mu},d\in\fD_{\la|\mu}}(-q)^{-l(y)}f_dT_xT_yT_d
=x_\la y_\mu\sum_{d\in\fD_{\la|\mu}} f_dT_d.$$ The converse inclusion is clear,
proving (2). For (3), consider the ``trace form''
$$\langle\,\,,\,\,\rangle:\sH\times \sH\lra\sZ,\quad \langle
a,b\rangle=\text{tr}(ab),$$ where tr$(\sum_wf_wT_w)=f_1$. A direct
computation shows that $\langle x_\la y_\mu T_u,T_vx_\la
y_\mu\rangle=\delta_{u,v^{-1}}\bsq^{l(u)}P_{\fS_{\la|\mu}}(\bsq)$,
where $P_{\fS_{\la|\mu}}(\bsq)$ is the Poincare Polynomial of
$\fS_{\la|\mu}$. Thus, we obtain a perfect paring 
$$(\,\,,\,\,):x_\la y_\mu\sH\times\sH x_\la y_\mu\longrightarrow\sZ,\quad (x_\la y_\mu T_u,T_vx_\la y_\mu)=\delta_{u,v^{-1}}\bsq^{l(u)}.$$
Now base change gives the required perfect paring for the isomorphism.
\end{proof}

 For a composition $\mu\models r$, let
$\tilde\mu$ be the partition obtained by rearranging the parts of
$\mu$. If $\mu\in \Lambda(m, r_1)$ with $r_1\le r$, define
$\tilde{\mu}^*=\tilde \mu\vee (1^{r-r_1})$. Then ${\tilde\mu}^*\in
\Lambda(m+r-r_1, r)$. The following result can be considered as the
 quantum version of \cite[Lem 3]{Ser}. Recall, for $\la\vdash r$,
the Specht module $S_\la$ of $\sH$  associated with $\la$ defined as a dual cell module in Theorem~\ref{RS}.

\begin{Prop}\label{YoungRule}
For any $\mu|\nu\in\Lambda(m|n,r)$, we have
 $$
 x_\mu y_\nu\sH_{\mathbb Q(\up)}\cong \bigoplus_{\la\in\Lambda^+(r)_{m|n},\la\trianglerighteq\tilde\mu^*}
m_{\la,\mu|\nu}S_{\la,\mathbb Q(\up)}.$$ If $
\tilde\mu^*\in\Lambda^+(r)_{m|n}$, then $S_{\tilde\mu^*,\mathbb
Q(\up)}$ is a direct summand of $
 x_\mu y_\nu\sH_{\mathbb Q(\up)}$ with multiplicity $1$.
\end{Prop}

\begin{proof}
Consider the basis $\{C_w\mid
\sR(w)\supseteq\mu\}$ and a left cell filtration for $\sH x_\mu$:
$$\sH x_\mu=E^\mu_0\supseteq E^\mu_1\supseteq\cdots\supseteq
E^\mu_{n_\mu-1}\supseteq E^\mu_{n_\mu}=0.$$ Since the set
$\{C_wy_\nu\mid \sR(w)\supseteq\mu\}\backslash\{0\}$ forms a basis
for $\sH x_\mu y_\nu$, this filtration induces a filtration of $\sH
x_\mu y_\nu$
$$\sH x_\mu y_\nu=E^{\mu|\nu}_0\supseteq E^{\mu|\nu}_1\supseteq\cdots\supseteq
E^{\mu|\nu}_{n_{\mu|\nu}-1}\supseteq E^{\mu|\nu}_{n_{\mu|\nu}}=0$$
with subfactors isomorphic to left cell modules. Applying Lemma
\ref{ABC} yields a dual left cell filtration for $x_\mu y_\nu\sH$:
$$0=F_{\mu|\nu}^0\subseteq F_{\mu|\nu}^1\subseteq\cdots\subseteq
F_{\mu|\nu}^{n_{\mu|\nu}}=x_\mu y_\nu\sH$$ where $F_{\mu|\nu}^j=(\sH
x_\mu y_\nu/E^{\mu|\nu}_j)^*$. The required isomorphism follows from
base change to $\mathbb Q(\up)$, Corollary \ref{multiplicity} and
Proposition \ref{RS}. The last equality follows from the fact that
$m_{{\tilde\mu}^*,\mu|\nu}=|\bfT^{sss}({\tilde\mu}^*,\mu|\nu)|=1$ if
$\tilde\mu^*\in\Lambda^+(r)_{m|n}$.
\end{proof}
\begin{rem} When
$\mu|\nu\in \Lambda(m|n, r)$ with $r-|\mu|>n$,
$S^{\tilde\mu^*}_{\mathbb Q(\up)}$ is not  a direct summand of $
 x_\mu y_\nu\sH_{\mathbb Q(\up)}$. However,  when $\mu\in\Lambda^+(r)_{m|n}$,
$S^\mu_{\mathbb Q(\up)}$ is a direct summand of
$x_{\mu'}y_{\mu''}\sH_{\mathbb Q(\up)}$ with multiplicity $1$ since
$|\bfT^{sss}(\mu,\mu'|\mu'')|=1$.
\end{rem}

For  $D=\jmath(\lambda|\mu,d, \xi|\eta)\in M(m|n, r)$ (see \eqref{sjmath}), we identify $D$ with the
 double coset $D=\mathfrak S_{\lambda|\mu}d\mathfrak
S_{\xi|\eta}$. Since $d\in \mathfrak D_{\lambda|\mu, \xi|\eta}\cap D$
is the shortest element in $D$, every $w\in D$ can be uniquely
written as $w=x.y.d.u.v$ with $x|y\in\fS_{\la|\mu}$ and
$u|v\in\fS_{\xi|\eta}\cap\fD_{\alpha|\beta}$, where
$\alpha=\alpha(D),\beta=\beta(D)$ are compositions of $|\xi|$ and
$|\eta|$, respectively, defined by
\begin{equation}\label{albe}
\fS_{\alpha^*}=d^{-1}\fS_{\la^*}d\cap\fS_{\xi^*}\,\,\text{ and
}\,\,\fS_{^*\!\beta}=d^{-1}\fS_{^*\!\mu}d\cap\fS_{^*\!\eta}.
\end{equation} By Lemma~\ref{TIP}, $\fS_{\alpha|\beta}=d^{-1}\fS_{\la|\mu}d\cap\fS_{\xi|\eta}=\fS_{\alpha^*}\fS_{^*\!\beta}$.
Define
\begin{equation}\label{T_D}
T_D=\sum_{u|v\in\fS_{\xi|\eta}\cap\fD_{\alpha|\beta}
}(-q)^{-\l(v)}x_\la y_\mu T_d T_uT_v.\end{equation} It is clear from
the definition that
$$T_D=x_\la y_\mu h_1=h_2x_\xi y_\eta=h_1'x_\la T_dy_\eta h_1''=h_2'x_\mu T_dy_\xi h_2''$$ for some $h_1,h_2,h_1',h_1'',h_2',h_2''\in\sH_R$.

\begin{rem} \label{TIP2}
The element $T_D$ is also defined for any $D=\jmath(\lambda|\mu,d, \xi|\eta)\in M(m+n, r)$.
However, if $d\in\fD_{\la|\mu,\xi|\eta}$ does not satisfy the two trivial intersection properties in Lemma~\ref{TIP}(3), then $T_D=0$.
\end{rem}
We will continue to make the following identification in the sequel.
\begin{equation}\label{fixed ro/co}
\aligned
M(m|n,r)_{\la|\mu,\xi|\eta}&:=\{D\in
M(m|n,r):\ro(D)=\la|\mu,\co(D)=\xi|\eta\}\\
&\,=\{D\in\mathfrak S_{\lambda|\mu}\backslash
\mathfrak S_{r}/\mathfrak
S_{\xi|\eta}:D\cap\fD_{\lambda|\mu,\xi|\eta}^{+,-}\not=\emptyset\}.\endaligned
\end{equation}

\begin{Prop} \label{curtis-gen} If $\mathcal
H_{\lambda|\mu,\xi|\eta}^{+,-}$ denotes the free $R$-submodule of
$\mathcal H_R$ spanned by $T_D$ for all $D\in M(m|n,r)_{\la|\mu,\xi|\eta}$, then
$$\begin{aligned} \mathcal H_{\lambda|\mu,\xi|\eta}^{+,-}&=x_\la y_\mu
\sH_R\cap\sH_Rx_\xi y_\eta\\&=
\{h\in \mathcal H\colon T_{s_1} h=h T_{t_1}=q h,  T_{s_2} h=h T_{t_2}=- h, \\
& \quad\,\qquad\quad \quad\forall s_1\in\lambda^*, s_2\in{}^*\!\mu, t_1\in\xi^*, t_2\in{}^*\!\eta\}.\\
\end{aligned}
$$
\end{Prop}

\begin{proof} When $\mu=\eta=(0)$, it is Curtis' result in
\cite{Curtis85}. In general, the proof is similar. We leave the
reader to verify $ T_{s_1} T_D=T_D T_{t_1}=q T_D,  T_{s_2} T_D=T_D
T_{t_2}=-  T_D$ for all $s_1\in\lambda^*, s_2\in{}^*\!\mu, t_1\in\xi^*, t_2\in{}^*\!\eta$.
This proves ``$\subseteq$" part of the result.

Conversely, Suppose  $h\in \mathcal H_R$ with $T_s h=q h$ for all $s\in\lambda^*$.  By \cite[1.9]{Curtis85}, we have
$a_w=a_{sw}$ for any $s\in \lambda^*$.
Similarly, we have $a_w=a_{ws}$ for any $s\in
\xi^*$. Therefore, $a_{w}=a_{y_1\cdot d\cdot y_2}$ if
$w=x_1\cdot x_2 \cdot d \cdot y_1\cdot y_2$ with $x_1\in \mathfrak
S_{\lambda^*}$, $x_2\in \mathfrak S_{^*\!\mu}$, $y_1\in \mathfrak S_{\xi^*}$ and
$y_2\in \mathfrak S_{^*\!\eta}$.
Similarly,  we have $a_{w}=-q a_{tw}$ (resp., $a_w=-qa_{wt}$) if $tw>w$ and
$t\in {}^*\mu$ (resp., $wt>w$ and $t\in {}^*\eta$). Consequently, for $w=x_1\cdot x_2 \cdot d \cdot y_1\cdot y_2$ as given above,
$a_{w}=(-q)^{-l(x_2)}a_{ d\cdot y_2}=a_d(-q)^{-(l(x_2)+l(y_2))}$. Hence, $h\in \mathcal H_{\lambda|
\mu,\xi|\eta}^{+,-}$.
\end{proof}

\begin{Defn}\label{ssch} Let $\fT(m|n,r;R)=\oplus_{\lambda|\mu\in \Lambda(m|n, r) } x_\lambda
y_\mu\mathcal H_R$ and
$$\sS(m|n, r;R):=\text{End}_{\mathcal H_R} (\fT(m|n,r;R))$$ and define a
$\mathbb Z_2$-grading by setting, for $i=0,1$,
\begin{equation}\label{grading for S}
 \sS(m|n,r)_i=\bigoplus_{{\lambda|\mu,\xi|\eta\in\Lambda(m|n,r)}
\atop {|\mu|+|\eta|\equiv i(\text{mod}2)}}\text{Hom}_{\mathcal
H_R}(x_\xi y_\eta\mathcal H_R,x_\la y_\mu\mathcal H_R).
\end{equation}
We call the $R$-algebra $\sS(m|n, r;R)$ with supermultiplication (see \eqref{supermult} below) the
\textit{quantum Schur superalgebra (or $q$-Schur superalgebras)}
over $R$. We will simply write $\sS(m|n, r)$ for $\sS(m|n, r;\sZ)$ and
$\fT(m|n,r)$ for $\fT(m|n,r;\sZ)$.\end{Defn}

 Note that there is also a $\mathbb
Z_2$-grading on $\fT(m|n,r;R)$ with
$$\fT(m|n,r;R)_0=\bigoplus_{{\lambda|\mu} \atop {|\mu|\equiv0(\text{mod}2)}}x_\lambda
y_\mu\mathcal H_R,\qquad \fT(m|n,r;R)_1=\bigoplus_{{\lambda|\mu} \atop
{|\mu|\equiv1(\text{mod}2)}}x_\lambda y_\mu\mathcal H_R.$$

\begin{rem} For the convenience of later use, our definition of $\sS(m|n, r)$ is taken
over the ring $\sZ=\mathbb Z[\up,\up^{-1}]$. However, it is clear that the quantum Schur
superalgebras is well defined over $\mathbb Z[\bsq,\bsq^{-1}]$, where $\bsq=\up^2$.
Thus, specializing $\bsq$ to $q\in R$ yields  the quantum Schur
superalgebras over $R$ (without assuming $\sqrt{q}$ exists in $R$).
\end{rem}

The quantum Schur superalgebras share some nice properties with the
quantum Schur algebras.

Recall the bijection introduced in \eqref{jmath--}. For $A=\jmath=(\lambda|\mu,d, \xi|\eta)\in M(m|n, r)$,  define $\phi_A=\phi_{\lambda|\mu,
\xi|\eta}^d\in\sS(m|n,r;R)$ by
\begin{equation}\label{standardbasis} \phi_{\lambda|\mu, \xi|\eta}^d
(x_\alpha y_\beta h) =\delta_{\xi| \eta, \alpha|\beta}
T_{\fS_{\la|\mu}d\fS_{\xi|\eta}} h,\end{equation} for all
$\alpha|\beta\in\Lambda(m|n,r)$ and $h\in\sH_R$. Clearly,
$\phi_{\lambda|\mu, \xi|\eta}^d\in\sS(m|n,r;R)_i$ if
$|\mu|+|\eta|\equiv i(\text{mod}\, 2)$.

\begin{thm}\label{base change} For any commutative ring $R$ which is a $\sZ$-module, the set
$\{\phi_A\mid A\in M(m|n,r)\}$ forms an $R$-basis for $\sS(m|n,r;R)$. In particular,
$\sS(m|n,r;R)\cong\sS(m|n,r)\otimes R$ has rank
 $$|M(m|n,r)|=\sum_{k=0}^r \begin{pmatrix} m^2+n^2+k-1\\k\\\end{pmatrix} \begin{pmatrix} 2mn\\r-k\\\end{pmatrix}. $$
 Moreover, there is an algebra anti-involution
$\tau:\sS(m|n,r;R)\to\sS(m|n,r;R)$ satisfying
$\tau(\phi_A)=\phi_{A^T}$, where $A^T$ denotes the transpose of $A$.
\end{thm}

\begin{proof} Since $\sS(m|n,r)=\bigoplus_{\lambda|\mu,\xi|\eta}\text{Hom}_{\mathcal H_R}(x_\xi y_\eta\mathcal H_R,x_\la y_\mu\mathcal
H_R),$ and $$\text{Hom}_{\mathcal H_R}(x_\xi y_\eta\mathcal
H_R,x_\la y_\mu\mathcal H_R)\cong  x_\la y_\mu\sH_R \cap \sH_Rx_\xi
y_\eta$$ as $R$-modules, the first assertion follows from
Proposition \ref{curtis-gen}. The rank assertion follows from a base change to a field by specializing $\up$ to 1 and
\cite[Th.4.18]{BR}.
The rest of the proof is clear.
\end{proof}
For  $A=\jmath(\lambda|\mu, d, \xi|\eta)$, by the ${\mathbb Z}_2$-grading \eqref{grading for S},
set $\hat A=|\mu|+|\eta|(\text{mod}\,2)$. Then the
supermultiplication is given by
\begin{equation}\label{supermult}
\phi_A\phi_B=(-1)^{\hat A\hat B}\phi_A\circ\phi_B,\quad\text{ for
all }A,B\in M(m|n,r).
\end{equation}
It is clear that the associativity holds with respect to the
supermultiplication. Moreover, it is clear $\phi_A T(m|n,r;R)_i\subseteq
T(m|n,r;R)_{\hat A+i}$ for all $A$. Hence, $T(m|n,r;R)$ is an
$\sS(m|n,r;R)$-supermodule.

\section{Canonical bases for quantum Schur superalgebras}

We now introduce canonical bases for quantum Schur superalgebras.
Recall the Kszhdan--Lusztig bases $\{C_w\}$ and $\{B_w\}$ for the Hecke algebra $\sH=\sH(\fS_r)$.

 For $D,D'\in M(m|n,r)_{\la|\mu,\xi|\eta}$ regarded as double cosets as in \eqref{fixed ro/co},
 let $w^+_D$ (resp. $w^-_D$) be the longest (resp., shortest) element in
$D$ and define
$$D\leq D'\text{ if and only if }w^+_D\leq w^+_{D'},$$
which is equivalent to $w^-_D\leq w^-_{D'}$; see, e.g., \cite[Lem.~4.35]{DDPW}.
Clearly,
$$l(w^+_D)=l(w_{0,\la})+l(w_{0,\mu})+l(d)+l(w_{0,\xi})-l(w_{0,\alpha})
+l(w_{0,\eta})-l(w_{0,\beta}),$$ where $d\in
D\cap\fD_{\la|\mu,\xi|\eta}$, and $\alpha,\beta$ are defined as in
\eqref{albe}. For $d\in \fD_{\la|\mu,\xi|\eta}$, let $d^*,{}^*\!d$ be defined as in Lemma~\ref{TIP} and
\begin{equation}\label{sT_D}
\sT_D=\up^{-l(d^*)}\up^{l({}^*\!d)-l(d)}T_D.
\end{equation}
If $D=\fS_{\la|\mu}$, then
$\sT_D=\up^{-l(w_{0,\la})}\up^{l(w_{0,\mu})}T_{\fS_{\la|\mu}}=(-1)^{l(w_{0,\mu})}
C_{w_{0,\la^*}}B_{w_{0,{}^*\!\mu}}$. Note that $l(w^+_D)=l(d^*)+l({}^*\!d)-\l(d)$.

\begin{Lemma} The restriction of the bar involution $^-$ on $\sH$ induces a bar involution
$^-$ on $\mathcal H_{\lambda|\mu,\xi|\eta}^{+,-}$. Moreover, for
$D,C\in M(m|n,r)_{\lambda|\mu,\xi|\eta}$, there exist
$r_{C,D}\in\sZ$ such that $r_{D,D}=1$ and
\begin{equation}\label{rcd}\overline{\sT}_D=\sum_{{C\in
M(m|n,r)_{\lambda|\mu,\xi|\eta}}\atop{C\leq
D}}r_{C,D}\sT_{C}.\end{equation}
\end{Lemma}
\begin{proof} Since $\bar x_\la=\bsq^{-l(w_{0,\la})}x_\la$ and $\bar y_\mu=
\bsq^{l(w_{0,\mu})}y_\mu$, it follows that $\overline{\sT}_D\in
\mathcal H_{\lambda|\mu,\xi|\eta}^{+,-}$. By Proposition
\ref{curtis-gen}, the restriction yields a bar involution on
$\mathcal H_{\lambda|\mu,\xi|\eta}^{+,-}$. On the other hand, since
\begin{equation}\label{rcd1}
\overline{\sT}_D=\up^{l(d^*)}\up^{-l({}^*\!d)+l(d)}\sum_{u|v\in\fS_{\xi|\eta}\cap\fD_{\alpha|\beta}}
(-\bsq)^{l(v)}\bar x_\la\bar y_\mu \bar T_d\bar T_u\bar T_v,
\end{equation}
and $T_s^{-1}=\bsq^{-1}T_s+(\bsq^{-1}-1)$, Proposition~\ref{curtis-gen} implies that
$\overline{\sT}_D$ can be written as in (\ref{rcd}).  It remains to prove that $r_{D,D}=1$. We write
$\overline{\sT}_D$ as a linear combination of $x_\la y_\mu T_z$,
$z\in\fD_{\la|\mu}$. By \eqref{rcd}, as the leading term of $\sT_D$, $x_\la y_\mu
T_dT_{w_{0,\xi}w_{0,\alpha}}T_{w_{0,\eta}w_{0,\beta}}$ has coefficient
$r_{D,D}\up^{-l(d^*)}\up^{l({}^*\!d)-l(d)}(-\bsq)^{-l(w_{0,\eta}w_{0,\beta})}$.
On the other hand, since
$$\aligned
& (-\bsq)^{l(w_{0,\eta}w_{0,\beta})}\bar x_\la\bar y_\mu \bar
T_d\bar T_{w_{0,\xi}w_{0,\alpha}}\bar
T_{w_{0,\eta}w_{0,\beta}}\\
=&\bsq^{-l(d^*)}\bsq^{l({}^*\!d)-l(d)}(-\bsq)^{-l(w_{0,\eta}w_{0,\beta})}
x_\la y_\mu  T_d T_{w_{0,\xi}w_{0,\alpha}}
T_{w_{0,\eta}w_{0,\beta}}+\text{lower
terms},\\
\endaligned$$
the same coefficient is equal by \eqref{rcd1} to $\up^{l(d^*)}\up^{-l({}^*\!d)+l(d)}\bsq^{-l(d^*)}\bsq^{l({}^*\!d)-l(d)}(-\bsq)^{-l(w_{0,\eta}w_{0,\beta})}$.
Hence, $r_{D,D}=1$.
\end{proof}

 By this lemma, a standard construction (see, e.g., \cite[\S0.5]{DDPW})
 gives the following.
\begin{Prop} There exists a unique $\sZ$-basis $\{\sfC_D\}_{D\in M(m|n,r)_{\lambda|\mu,\xi|\eta}}$
for $\mathcal H_{\lambda|\mu,\xi|\eta}^{+,-}$ such that
$\bar\sfC_D=\sfC_D$ and $\sfC_D=\sum_{C\leq D}p_{C,D}\sT_{C}$,
where $p_{D,D}=1$ and $p_{C,D}\in\up^{-1}{\mathbb Z}[\up^{-1}]$ if
$C<D$. Moreover, if $D=\fS_{\la|\mu}$, then
$\sfC_D=\sT_{\fS_{\la|\mu}}$.
\end{Prop}

For any $D\in M(m|n,r)$, if we put
$$\vph_D=\up^{-l(d^*)}\up^{l({}^*\!d)-l(d)}\up^{-l(w_{0,\xi^*})+l(w_{0,{}^*\!\eta})}\phi_D,$$
where $\co(D)=\xi|\eta$ and $\phi_D$ is defined in
(\ref{standardbasis}), then
$$\vph_D(\sT_{\fS_{\alpha|\beta}})=\delta_{\co(D),\alpha|\beta}\sT_D;$$
cf. \eqref{sT_D}. We now have the following.

\begin{thm}\label{canonical} The bar involution $^-:\sZ\to\sZ$ can be extended to a ring homomorphism
$^-:\sS(m|n,r)\to\sS(m|n,r)$ defined by
linearly extending the action:
$$\bar\vph_D=\sum_{C}r_{C,D}\vph_{C},$$
where the scalars $r_{C,D}$'s are defined in (\ref{rcd}). In
particular, there is a unique basis $\{\Theta_D\}_{D\in M(m|n,r)}$
satisfying
$$\bar\Theta_D=\Theta_D,\,\,\,\Theta_D-\vph_D\in\sum_{C<D}\up^{-1}{\mathbb Z}[\up^{-1}]\vph_{C}.$$
\end{thm}

\begin{proof} We first observe that
$\bar\vph_D(\sT_{\fS_{\co(D)}})=\overline{\vph_D(\sT_{\fS_{\co(D)}})}$
and the bar involution preserves the $\mathbb Z_2$-grading. Thus, for
$C,D\in M(m|n,r)$, $\vph_C\vph_D=(-1)^{\hat C\hat
D}\vph_C\circ\vph_D$ and $\bar\vph_C\bar\vph_D=(-1)^{\hat C\hat
D}\bar\vph_C\circ\bar\vph_D$. Hence, to prove that the bar
involution is a ring homomorphism, it suffices to prove that
$\overline{\vph_C\circ\vph_D}=\bar\vph_C\circ\bar\vph_D,$ for all
$C,D$ with $\co(C)=\ro(D)$. This is clear since
$$\aligned
\bar\vph_C\circ\bar\vph_D(\sT_{\fS_{\co(D)}})&=\bar\vph_C(\overline{\sT_D})\\
&=\bar\vph_C(\bar\sT_{\fS_{\co(D)}}\,\bar h_D)\,\,\text{ where
}\sT_D=\sT_{\fS_{\co(D)}}h_D,\\
&=\overline{\sT_C}\,\bar h_D=\overline{\sT_Ch_D}\qquad\qquad\text{ since }\bar\sT_{\fS_{\co(D)}}=\sT_{\fS_{\co(D)}}\\
&=\overline{\vph_C\circ\vph_D(\sT_{\fS_{\co(D)}})}\\
&=\overline{\vph_C\circ\vph_D}(\sT_{\fS_{\co(D)}}),
\endaligned
$$
proving the first assertion. For the last assertion, the
construction of the basis is standard; see, e.g.,
\cite[\S0.5]{DDPW}.
\end{proof}
Note that $\Theta_D$ is the element satisfying
$\Theta_D(\sT_{\fS_{\co(D)}})=\sfC_D$ and, if
$D=\jmath(\la|\mu,1,\la|\mu)$, then $\Theta_D=\vph_D$ is an
idempotent.

The basis $\{\Theta_D\}$ does not seem to have a direct connection
with the canonical bases $\{C_w\}$ for Hecke algebras. However,
there is a $\mathbb Q(\up)$-basis which is defined via the
$C$-basis.

Let $y'_\mu=\up^{l(w_{0,\mu})}y_\mu=(-1)^{l(w_{0,
\mu})}B_{w_{0,\mu}}$ so that $\bar y'_\mu=y'_\mu$. For
$D=\jmath(\la|\mu,d,\xi|\eta)$, let
$$\sT'_D=y'_\mu\sT_{D^*}y'_\eta,$$
where $D^*=\fS_{\la^*}d\fS_{\xi^*}$ and
$\sT_{D^*}=\up^{-l(d^*)}\sum_{x\in D^*}T_x$. Clearly,
\begin{equation}\label{TandT'}
\sT'_D=\up^{l(w_{0,\beta})}P_{\fS_\beta}(\bsq^{-1})\sT_D,
\end{equation}
 where
$P_{\fS_\beta}(\bsq)$ is the Poincar\'e polynomial of $\fS_{\beta}$; see
\eqref{albe} for the definition of $\beta=\beta(D)$.
This is because, by the definitions of \eqref{T_D} and \eqref{sT_D} of $T_D$ and $\sT_D$,
\begin{equation}\label{T_D*}
\sT_D=\sum_{v\in\fS_{{}^*\!\eta}\cap\fD_{{}^*\!\beta}
}\up^{l({}^*\!d)-l(d)}y_\mu \sT_{D^*} (-q)^{-\l(v)}T_v=\frac {\up^{l({}^*\!d)-l(d)}}{P_{\fS_\beta}(\bsq^{-1})}y_\mu \sT_{D^*}y_\eta,
\end{equation}
and $l({}^*\!d)=l(w_{0,\mu})+l(w_{0,\eta})-l(w_{0,\beta})$.

Let
$\fH^{+,-}_{\la|\mu,\xi|\eta}$ be the $\sZ$-span of $\sT'_D$, $D\in
M(m|n,r)_{\la|\mu,\xi|\eta}$. This is a $\sZ$-submodule of
$\sH^{+,-}_{\la|\mu,\xi|\eta}$ satisfying
$$\fH^{+,-}_{\la|\mu,\xi|\eta}\otimes\mathbb Q(\up)=\sH^{+,-}_{\la|\mu,\xi|\eta}\otimes\mathbb Q(\up).$$

\begin{Prop} \label{sfC'basis}For any $C,D\in
M(m|n,r)_{\la|\mu,\xi|\eta}$, there exist $r^*_{C,D}\in\sZ$ such
that $r^*_{D,D}=1$ and
$$\bar \sT'_D=\sum_{{C\in
M(m|n,r)_{\la|\mu,\xi|\eta}}\atop{C\leq D}}r^*_{C,D}
\sT_{C}'.$$ Moreover, if $\{\sfC'_D\}_{D\in
M(m|n,r)_{\la|\mu,\xi|\eta}}$ denotes the associated canonical basis
for $\fH^{+,-}_{\la|\mu,\xi|\eta}$, then $\sfC'_D:=y'_\mu
C_{d^*}y'_\eta$, where $d^*$ is the longest element in $D^*$.
\end{Prop}
\begin{proof} Let $D=\jmath(\la|\mu,d,\xi|\eta)$.
Since  $\bar
\sT_{D^*}=\sum_{B^*\in\fS_{\la^*}\backslash\fS_r/\fS_{\xi^*}}r^*_{B^*,D^*}\sT_{B^*}$,
it follows that
$$\bar \sT'_D=y'_\mu\bar \sT_{D^*}y'_\eta
=\sum_{B^*\in\fS_{\la^*}\backslash\fS_r/\fS_{\xi^*}}r^*_{B^*,D^*}y'_\mu\sT_{B^*}y'_\eta.$$
Here, $B^*\leq D^*$. In other words, if $d_{B^*}$ denotes the shortest element in $B^*$, then $d_{B}\leq d$.
By Remark~\ref{TIP2}, $y'_\mu\sT_{B^*}y'_\eta\neq0$ implies that $f(B^*):=\fS_{\la|\mu}d_{B^*}\fS_{\xi|\mu}=\fS_{\la|\mu}d_{B}\fS_{\xi|\mu}
$ for some $d_B\in\fD_{\la|\mu,\xi|\mu}^{+,-}$ and $d_B\leq d_{B^*}\leq d$, and hence, $f(B^*)\in
M(m|n,r)_{\la|\mu,\xi|\eta}$ and $f(B^*)\leq D$. Thus, if $C\in M(m|n,r)_{\la|\mu,\xi|\eta}$ and define $r^*_{C,D}=\sum_{B^*, f(B^*)=C}r^*_{B^*,D^*}$, then $r^*_{D,D}=r^*_{D^*,D^*}=1$.
This proves the first assertion.

On the other
hand, $C_{d^*}=\sT_{D^*}+\sum_{C^*<D^*}p_{C^*,D^*}\sT_{C^*}$ for
some $p_{C^*,D^*}\in\up^{-1}\mathbb Z[\up^{-1}]$. Putting $b_D=y'_\mu C_{d^*}y'_\eta$, we have
$\overline{b_D}=b_D$, and a similar argument shows that
$b_D=\sum_{C\leq D}p_{C,D}\sT_C'$ where $C\in
M(m|n,r)_{\la|\mu,\xi|\eta}$, $p_{D,D}=1$ and $p_{C,D}\in\up^{-1}\mathbb Z[\up^{-1}]$ for $C<D$. Now,
 the
uniqueness of the canonical basis forces $\sfC'_D=b_D=y'_\mu
C_{d^*}y'_\eta$, proving the last statement.
\end{proof}

 Taking bar involution on both sides of \eqref{TandT'}, we obtain the following
 relation on the entries of ``$R$-matrices'' $(r_{C,D})$ and $(r_{C,D}^*)$:
\begin{equation}\label{rr*}
r^*_{C,D}=r_{C,D}\frac{\up^{-l(w_{0,\beta(D)})}P_{\fS_{\beta(D)}}(\bsq)}{\up^{l(w_{0,\beta(C)})}
P_{\fS_{\beta(C)}}(\bsq^{-1})}\quad\text{ for all }C,D\in M(m|n,r).
\end{equation}
Thus, no obvious relation between the $\sfC$-basis and $\sfC'$-basis is seen. However, when restrict to the tensor space,
the two bases coincide.

\begin{rem}\label{om1om2}
If $m+n\geq r$, then there exist unique $\omega_1|\omega_2\in\Lambda(m|n,r)$ such that $$\omega_1\vee\omega_2=\omega:=(\underbrace{1,\ldots,1}_{r\text{ times}},0,\ldots).$$
Thus, if $\xi|\eta=\omega_1|\omega_2$,
then $\fS_\beta=\{1\}$ and hence, $\fH^{+,-}_{\la|\mu,\omega_1|\omega_2}=\sH^{+,-}_{\la|\mu,\omega_1|\omega_2}$ and
$r^*_{C,D}=r_{C,D}$. Consequently, $\sfC_D=\sfC'_D$ in this case.
Thus, if we put
$$M(m|n,r)_{\text{\rm tsp}}=\begin{cases}
\{A\in M(m|n,r)\mid \co(A)=\omega_1|\omega_2\},&\text{ if }m+n\geq r,\\
\{A\in M(m'|n',r)\mid \ro(A)\in\Lambda(m|n,r),\co(A)=\omega_1|\omega_2\},&\text{ if }m+n<r,
\end{cases}$$
where $m\leq m'$, $n\leq n'$ and $m'+n'\geq r$, then
\begin{equation}\label{CBforTS}
\{\sfC_D\mid D\in M(m|n,r)_{\text{\rm tsp}}\}=\{\sfC_D'\mid D\in M(m|n,r)_{\text{\rm tsp}}\}
\end{equation}
forms a basis for $\fT(m|n,r)$. Call it the canonical basis of $\fT(m|n,r)$.
\end{rem}

Let $\bsS(m|n,r)=\sS(m|n,r)\otimes\mathbb Q(\up)$. For every $D=\jmath(\la|\mu,d,\xi|\eta)\in
M(m|n,r)$, define
$\Theta_D'\in\sS(m|n,r)$ by setting
\begin{equation}\label{Theta'}
\Theta_D'(x'_\xi y'_\eta)=\sfC_D'=y_\mu'C_{d^*}y_\eta',
\end{equation}
where $x_\xi'=C_{w_{0,\xi}}=\up^{-l(w_{0,\xi})}x_\xi$.

\begin{Cor} The set $\{\Theta_D'\}_{D\in M(m|n,r)}$ forms a $\mathbb Q(\up)$-basis
for $\bsS(m|n,r)$.
\end{Cor}

We will prove by using cell theory that  this basis gives rise to all simple modules of
$\bsS(m|n,r)$  in Section~7. Such a result can be considered as a
generalization of \cite[Theorem~1.4]{KL79}.


\section{Supercells and their associated cell representations}

We now use the basis $\{\Theta_D'\}_D$ given at the end of \S6 to
construct irreducible representations of $\bsS(m|n,r)$. Recall the
map defined in \eqref{sjmath}. We also write $\lambda|\mu=\xi|\eta$
if $\lambda=\xi$ and $\mu=\eta$.

\begin{Defn}\label{CELLS} For $A,B\in M(m|n,r)$ with
$A=\jmath^{+,-}(\alpha|\beta,y,\gamma|\delta)$ and $
B=\jmath^{+,-}(\la|\mu,w,\xi|\eta)$, define
$$A\leq_L B\iff y\leq_L w\text{ and } \xi|\eta=\gamma|\delta \,\,(\text{or }\co(A)=\co(B)).$$

Define $A\leq_RB$ if $A^T\leq_LB^T$. Let $\leq_{LR}$ be the preorder
generated by $\leq_L$ and $\leq_R$. The relations give rise to three
equivalence relations $\sim_L,\sim_R$ and $\sim_{LR}$. Thus,
$A\sim_XB$ if and only if $A\leq_XB\le_XA$ for all $X\in\{L,R,LR\}$.
The corresponding equivalence classes in $M(m|n,r)$ with respect to
$\sim_L,\sim_R$ and $\sim_{LR}$ are called {\it left cells, right
cells} and {\it two-sided cells}, respectively.
\end{Defn}
In particular, for $A,B$ as above, we have
\begin{itemize}
\item[(1)] $A\sim_L B\iff y\sim_Lw\text{ and }
\xi|\eta=\gamma|\delta$;
\item[(2)] $A\sim_R B\iff y\sim_Rw\text{ and }
\la|\mu=\alpha|\beta$;
\item[(3)] $A\le_LB$ and $A\sim_{LR}B$ $\iff$ $A\sim_LB$;
\item[(4)] $A\le_RB$ and $A\sim_{LR}B$ $\iff$ $A\sim_RB$;
\end{itemize}
Statements (3) and (4) follows from the fact that if $y\le_L w$ and $y\sim_{LR}w$ then $y\sim_Lw$;
see \cite[Cor.~6.3(c)]{Lcell1}.

\begin{Lemma}\label{tri-rel} For $A,B\in M(m|n,r)$, if $\Theta'_A \Theta'_B=\sum_{C\in M(m|n,r)} f_{A,B,C} \Theta'_C$,
then $f_{A,B,C}\neq 0$ implies  $C\le_LB$ and $C\le_RA$.
\end{Lemma}

\begin{proof}Let $A=\jmath^{+,-}(\alpha|\beta,y,\gamma|\delta)$ and
$B=\jmath^{+,-}(\la|\mu,w,\xi|\eta)$.
 If $\lambda|\mu\neq \gamma|
\delta$, then $f_{A,B,C}=0$ for all $C$. Suppose $\lambda| \mu=
\gamma|\delta$ and let $h_\lambda\in \mathbb Z[v, v^{-1}]$ be defined
by $x_\lambda' x_\lambda'=h_\lambda x_\lambda'$. We have by \eqref{Theta'}
$$
\aligned
 \Theta'_A \Theta'_B (x_\xi' y_\eta')&=\Theta'_A(y_\mu' C_w y_\eta')=h_\la^{-1}\Theta'_A(x_\la'y_\mu')C_w y_\eta'=h_\la^{-1}y_\beta' C_yy_\mu'C_w y_\eta'\\
&= \sum_{z\in\fD_{\alpha|\beta, \xi|\eta}^{+,
-}} h_{\lambda}^{-1}h_{y,w,z}
y_\beta' C_z y_\eta' =\sum_{z\in \fD_{\alpha|\beta, \xi|\eta}^{+,
-}}h_{\lambda}^{-1}h_{y,w,z}\Theta'_C(x_\xi' y_\eta'),\endaligned$$ where
$h_{y,w,z}\in \mathbb Z[v, v^{-1}]$ satisfy $C_yy_\mu'
C_w=\sum_zh_{y,w,z} C_z,$ and
$C=\jmath^{+,-}(\alpha|\beta,z,\xi|\eta)$. Here we have used the fact that $
y_\beta' C_z y_\eta'\neq 0\implies z\in \fD_{\alpha|\beta, \xi|\eta}^{+, -}$
(see Remark \ref{TIP2}). Hence,
$$f_{A,B,C}=\begin{cases}h_\lambda^{-1}h_{y,w,z}, &\text{ if } y_\beta' C_z y_\eta'\neq 0,\\
0,&\text{ otherwise.}\end{cases}$$ Since $h_{y,w,z}\neq 0$ implies $z\le_L
w$, it follows that
$f_{A,B,C}\not=0$ implies $z\le_Lw,\co(C)=\co(B),$ proving the first assertion.
The second assertion follows from the anti-involution $\tau$ given in Theorem
\ref{base change}.
\end{proof}

For each $A\in M(m|n, r)$, let $(\ssS(A), \ssT(A))$ be the image of $A$
under the  RSK super-correspondence in  Theorem~\ref{RSKs}. The
following result can be considered as a generalization of
Theorem~\ref{RS}(1)--(3).

\begin{Lemma}\label{GRS} Suppose $A,B\in M(m|n,r)$. Then
\begin{itemize}
\item[(1)]  $A\sim_L B$
 if and only if $\ssT(A)=\ssT(B)$.

\item[(2)] $A\sim_R B$ if and only if $\ssS(A)=\ssS(B)$.

\item[(3)] $A\sim_{LR}B$ if and only if $\ssT(A)$, $\ssT(B)$ have the same shape.
\end{itemize}
\end{Lemma}
\begin{proof}
Suppose  $w_1\in \mathfrak D_{\lambda|\mu, \xi|\eta}^{+, -}$
and $w_2\in \mathfrak D_{\alpha|\beta, \gamma|\delta}^{+, -}$ which have images $(\ssS_{w_1},\ssT_{w_1})$ and $(\ssS_{w_2},\ssT_{w_2})$ under the map $\partial$ defined in \eqref{genrs}. By the constructions of $\partial$ and its inverse (see proof of Theorem~\ref{RSKs}), we sees easily the following:
\begin{enumerate}\item[(1)]  $w_1\sim_L w_2$ and $\xi|\eta=\gamma|\delta$
 if and only if $\ssT_{w_1}=\ssT_{w_2}$.

\item[(2)] $w_1\sim_R w_2$ and $\lambda|\mu=\alpha|\beta$ if and only if $\ssS_{w_1}=\ssS_{w_2}$.

\item[(3)] $w_1\sim_{LR} w_2$ if and only if $\ssT_{w_1}$, $\ssT_{w_2}$ have the same shape.
\end{enumerate}
Now the assertions follow immediately.
\end{proof}

For $\nu\in \Lambda^+(r)_{m|n}$, let
$$I(\nu)= \bigcup_{\lambda|\mu\in \Lambda(m|n, r)}
\bfT^{sss}(\nu, \lambda|\mu).$$
By the RSK super-correspondence, if $A\RSKs(\ssS,\ssT)\in I(\nu)$, we relabel the basis element $\Theta'_A$  as
 $$
\Theta^{\prime\,\nu}_{\ssS,\ssT}:=\Theta'_A.$$

\begin{Prop}\label{cellular} The $\mathbb
Q(\up)$-basis for $\bsS(m|n,r)$
$$\{\Theta^{\prime\,\nu}_{\ssS,\ssT}\mid \nu\in \Lambda^+(r)_{m|n},\ssS,\ssT\in I(\nu)\}
=\{\Theta'_A\mid A\in M(m|n,r)\}$$ is a cellular basis in the sense of \cite{GL}.
\end{Prop}
\begin{proof}  Recall from \cite{GL} the ingredients for a cellular basis. We have
a poset $\Lambda^+(r)_{m|n}$ together with the dominance order $\trianglerighteq$, index sets
$I(\nu)$ of the basis, and an anti-involution $\tau$ satisfying $\tau(\Theta^{\prime\,\nu}_{\ssS,\ssT})= \Theta^{\prime\,\nu}_{\ssT,\ssS}$ by Theorems \ref{base change} and \ref{RSKs}. It remains to check the triangular relations.

Let $\bsS(m|n, r)^{\rhd \nu}$ be the $\mathbb Q(\up)$-subspace
spanned by  $\Theta^{\prime\alpha}_{\ssS_1,\ssT_1}$ for all $\alpha\rhd \nu$ and
$\ssS_1, \ssT_1\in I(\alpha)$.   For
$\la,\nu\in \Lambda^+(r)_{m|n}$ and $\ssS,\ssT\in
I(\la),\ssS',\ssT'\in I(\nu)$, Lemmas \ref{tri-rel} and \ref{GRS}(3) imply that:
$$\Theta^{\prime\la}_{\ssS,\ssT}\Theta^{\prime\,\nu}_{\ssS',\ssT'}\equiv\sum_{C\in M(m|n,r),C\sim_LB}f_{A,B,C}\Theta^{\prime\,\nu}_C\,(\text{mod}\,
\bsS(m|n, r)^{\rhd \nu}),$$
where $A\RSKs(\ssS,\ssT)$, $B\RSKs(\ssS',\ssT')$ and $C\RSKs(\ssS'',\ssT'')$.
Since $C\sim_L B$, it follows from \ref{GRS}(1), $\ssT''=\ssT'$. If $\la\triangleright \nu$, then all $f_{A,B,C}=0$.
If $\la=\nu$ and $f_{A,B,C}\neq0$, then $C\sim_RA$. Hence, $\ssS''=\ssS$ and $f_{A,B,C}=f(\ssT,\ssS')$ is independent of $\ssT'$. Finally, if
$\la\triangleleft \nu$ and $f_{A,B,C}\neq0$, then $f_{A,B,C}=f(\ssS,\ssT,\ssS')$ is also independent of $\ssT'$, as required.
\end{proof}

 For each $\nu\in \Lambda^+(r)_{m|n}$ and $\ssT\in I(\nu)$, let
\begin{equation}\label{sim}\bsDe(\nu)_\ssT=\bsS(m|n, r)^{\trianglerighteq \nu,\ssT}/\bsS(m|n, r)^{\rhd \nu},
\end{equation}
where $\bsS(m|n, r)^{\trianglerighteq \nu,\ssT}$ is the $\mathbb Q(\up)$-space spanned by $\bsS(m|n, r)^{\rhd \nu}$ and $\Theta^{\prime\,\nu}_{\ssS,\ssT}$, $\ssS\in I(\nu)$. These are called {\it left cell modules}.
Let
 $\ssT_\nu$ be the unique element in $\bfT^{sss}(\nu,
\nu'|\nu'')$ as described in Example \ref{ssT_la} and let $\bsDe(\nu)=\bsDe(\nu)_{\ssT_\nu}$.
The following result generalizes the second part of Theorem \ref{RS}.

\begin{thm} For each $\nu\in \Lambda^+(r)_{m|n}$ and $\ssT\in I(\nu)$, we have $\bsDe(\nu)_\ssT\cong\bsDe(\nu)$ as $\bsS(m|n,
r)$-supermodules. Moreover, the set $\{\bsDe(\nu)\mid \nu\in \Lambda^+(r)_{m|n}\}$ is a
complete set of pair-wise non-isomorphic irreducible $\bsS(m|n,
r)$-supermodules.
\end{thm}
\begin{proof} The first assertion follows from the cellular property.
Thus, Proposition \ref{cellular} implies $\dim \bsS(m|n,
r)=\sum_{\nu\in \Lambda^+(r)_{m|n}} (\dim \bsDe(\nu))^2$. Since $\up$ is an indeterminate, $\bsH$ is semisimple.
Hence, $\bsS(m|n, r)$ is also semisimple
 as the super product does
not change the radical of the endomorphism algebra with a usual
product. By the Wedderburn-Artin Theorem,  $\{\bsDe(\nu)\mid \nu\in
\Lambda^+(r)_{m|n}\}$ is a complete set of pair-wise non-isomorphic
irreducible $\bsS(m|n, r)$-modules. Finally, it is routine to check
that $\bsDe(\nu)$'s  are $\bsS(m|n, r)$-supermodules. In fact, they
are the absolute irreducible supermodues in the sense of
\cite[2.8]{BK:sergeev}.\end{proof}

We end this section with a second look at the canonical basis for $\fT(m|n,r)$ described in
Remark \ref{om1om2}. Recall the $\phi$-basis defined in \eqref{standardbasis}.

\begin{Lemma} If $m+n\geq r$, then there is an $\sS(m|n,r)$-$\sH$-bimodule isomorphism
between $\sS(m|n,r)\phi_{\omega_1|\omega_2}$ and $\fT(m|n,r)$,
where $\omega_1,\omega_2$ are defined in \ref{om1om2}, and $\phi_{\omega_1|\omega_2}:=\phi_{\omega_1|\omega_2,\omega_1|\omega_2}^1.$
\end{Lemma}

\begin{proof} Consider the evaluation map
$$ev:\sS(m|n,r)\phi_{\omega_1|\omega_2}\overset\sim\longrightarrow \fT(m|n,r),\,\,
\phi\longmapsto\phi(1),$$ which is clearly an $\sS(m|n,r)$-$\sH$-bimodule isomorphism.
\end{proof}

If $m+n<r$, we choose $m',n'$ with $m\leq m'$, $n\leq n'$ and $m'+n'\geq r$. Then $\Lambda(m|n,r)$ can be regarded as a subset
of $\Lambda(m'|n',r)$. Let $e=\sum_{\la|\mu\in\Lambda(m|n,r)}\phi_{\text{diag}(\la|\mu)}$. Then
$\sS(m|n,r)\cong e\sS(m'|n',r)e$ is a centralizer subalgebra of $\sS(m'|n',r)$, and the map $ev$
above induces $\sS(m|n,r)$-$\sH$-bimodule isomorphism $e\sS(m'|n',r)\phi_{\omega_1|\omega_2}\cong
\fT(m|n,r)$.

Let
$$\aligned
\sE(m|n,r)&=\begin{cases}\sS(m|n,r)\phi_{\omega_1|\omega_2},&\text{ if }m+n\geq r,\\
e\sS(m'|n',r)\phi_{\omega_1|\omega_2},&\text{ if }m+n<r,
\end{cases}\\
\endaligned
$$
where $m\leq m'$, $n\leq n'$ and $m'+n'\geq r$. By the lemma and Remark \ref{om1om2},
we have the following.

\begin{Prop} By identifying $\sE(m|n,r)$ with $\fT(m|n,r)$,  the basis \eqref{CBforTS} for $\fT(m|n,r)$
identifies the basis
$$\{\Theta_D=\Theta_D'\mid A\in M(m|n,r)_{\text{\rm tsp}}\}$$(for $\sE(m|n,r)$), which is canonically related (in the sense of Theorem \ref{canonical}) to the standard basis $\{\varphi_A\mid A\in M(m|n,r)_{\text{\rm tsp}}\}$ for $\sE(m|n,r)$.
\end{Prop}

By definition, $\bfT^{sss}(\nu, \omega_1|\omega_2)=\bfT^{s}(\nu)$. Thus,
by \ref{RSKs}, restriction gives a bijection:
$$M(m|n,r)_{\text{tsp}}\longrightarrow\bigcup_{\lambda|\mu\in \Lambda(m|n, r)\atop\nu\in \Lambda^+(r)_{m|n}
 }  \bfT^{sss}(\nu, \lambda|\mu)\times \bfT^{s}(\nu).$$

Fix a linear ordering on
$\Lambda^+(r)_{m|n}=\{\nu^{(1)},\nu^{(2)},\ldots,\nu^{(N)}\}$ which
refines the opposite dominance ordering $\trianglerighteq$, i.~e.,
 $\nu^{(i)}\triangleright\nu^{(j)}$ implies $i<j$.
For each $1\le i\le N$, let $\sE_i$ denote the $\sZ$-free
submodule of $\sE(m|n,r)$ spanned by all $\Theta_{\ssS,\ts}^{\nu^{(i)}}$ with
$(\ssS,\ts)\in\bfT^{sss}(\nu^{(i)}, \lambda|\mu)\times \bfT^{s}(\nu^{(i)})$. Then we obtain a
filtration by $\sS(m|n,r)$-$\sH$-subbimodules:
\begin{equation}\label{D-filtration}
0=\sE_0\subseteq
\sE_1\subseteq\cdots\subseteq\sE_N=\sE(m|n,r).\end{equation}

Let $\bsE_i=\sE_i\otimes\mathbb Q(\up)$.
By the
cellular property established in Proposition \ref{cellular}, each
section $\bsE_i/\bsE_{i-1}$ is isomorphic to a direct sum of $|\bfT^s(\nu^{(i)})|$ copies of
 left
cell modules $\bsDe(\nu^{(i)})$ and to a direct sum of $|I(\nu^{(i)})|$ copies of right (cell) $\bsH$-modules\footnote{Note that the right cell module $S^\la$ defined by the right cell containing $w_{0,\la}$ is a homomorphic image of $x_\la\sH$, while the dual left cell module $S_\la$ defined in \ref{RS} is a submodule of $x_\la\sH$.}
$S^{\nu^{(i)}}_{\mathbb Q(\up)}.$ Hence,  $\bsE_i/\bsE_{i-1}\cong
\bsDe(\nu^{(i)})\otimes S^{\nu^{(i)}}_{\mathbb Q(\up)}$ as $\bsS(m|n,r)$-$\bsH$-bimodules.

\begin{Cor}\label{decomp} There is an $\bsS(m|n,r)$-$\bsH$-bimodule decomposition:
$$\bsE(m|n,r)\cong\bigoplus_{\nu\in\Lambda^+(r)_{m|n}}\bsDe(\nu)\otimes S^{\nu}_{\mathbb Q(\up)}.$$
\end{Cor}


 \begin{rem} For $\nu=\nu^{(i)}$, let $L(\nu)$ be the submodule of $\sE_i/\sE_{i-1}$ spanned
 by all $\Theta_{\ssS,\ts^\nu}^\nu$.
 Since $\Theta_{\ssS,\ts^\nu}^\nu=\Theta_{\ssS,\ts^\nu}^{\prime\,\nu}$ by Remark \ref{om1om2},
 one checks directly that $L(\nu)$ is an $\sS(m|n,r)$-module. In other words, $L(\nu)$ is closed under the action of the canonical basis $\Theta_A$.
 Base change allows us to investigate representations at roots of unity.
 We hope to classify  the irreducible $\sS(m|n, r)_R$-supermodules
 elsewhere when $\up^2$ is specialized to a root of unity in a field $R$.
\end{rem}

\section{A super analogue of the quantum Schur--Weyl reciprocity}

In this section, we first establish a double centralizer property. Then
we prove that the algebra $\sS(m|n,r)$ is isomorphic to the
endomorphism algebra of a tensor space considered in \cite{Mit}.
Thus, we reproduced the super analogue of the quantum Schur--Weyl reciprocity established in \cite{Mit}.

Let $\bsT(m|n,r)=\fT(m|n,r)\otimes{\mathbb Q(\up)}$.
\begin{thm}\label{DCP} The $\bsS(m|n,r)$-$\bsH$-bimodule structure $\bsT(m|n,r)$ satisfies the following double centralizer property
$$\bsS=\text{\rm End}_{\bar\bsH}(\bsT(m|n,r))\text{ and }\bar\bsH=\text{\rm End}_{\bsS}(\bsT(m|n,r)),$$
where $\bar\bsH$ is the image of $\bsH$ in
$\text{\rm End}_{\mathbb Q(\up)}(\bsT(m|n,r))$ and $\bsS=\bsS(m|n,r)$. Moreover, there is a category
equivalence
$$\text{\rm Hom}_{\bar\bsH}(-,\bsT(m|n,r)):\text{\bf mod-}\bar\bsH\longrightarrow\bsS\text{-\bf mod}
.$$
\end{thm}

\begin{proof} First, as a quotient of a semisimple algebra, $\bar\bsH$ is semisimple.
By Corollary \ref{decomp}, 
$S^{\nu}_{\mathbb Q(\up)}$, $\nu\in\Lambda^+(r)_{m|n}$,  are non-isomorphic
irreducible $\bar\bsH$-modules. Thus, dim$\bar\bsH\geq
d:=\sum_{\nu\in\Lambda^+(r)_{m|n}}(\text{dim}S^{\nu}_{\mathbb
Q(\up)})^2.$ On the other hand, Corollary \ref{decomp} implies that
dim\,$\text{\rm End}_{\bsS}(\bsT(m|n,r))=d$. Hence, a dimensional
comparison forces $\bar\bsH=\text{\rm End}_{\bsS}(\bsT(m|n,r))$. The
rest of the proof is clear by noting that the inverse functor of
$\text{\rm Hom}_{\bar\bsH}(-,\bsT(m|n,r))$ is $\text{\rm
Hom}_\bsS(-,\bsT(m|n,r))$.
\end{proof}

We now relate the quantum Schur superalgebras with the quantum enveloping superalgebra $\mathbf U_\up^\sigma(\mathfrak{gl}(m|n))$.
We use the quantum superspace $V(m|n)$ considered in \cite{Manin} and \cite{Mit}.

Let $V(m|n)$ be a free $\sZ$-module of rank $m+n$ with basis $e_1, e_2,
\ldots, e_{m+n}$. The map by setting $\hat i=0$ if $1\le i\le m$, and $\hat i=1$
otherwise, as given in \eqref{hatmap} yields a $\mathbb Z_2$-grading on $V(m|n)=V_0\oplus V_1$ where
$V_0$ is spanned by $e_1, e_2,
\ldots, e_m$ and $V_1$ by $e_{m+1}, e_{m+2},
\cdots, e_{m+n}$. Thus, $V(m|n)$ becomes a ``superspace''.

Let $\check \sR: V(m|n)^{\otimes 2} \rightarrow  V(m|n)^{\otimes2}$ be defined by
\begin{equation}\label{rmatrix}(e_c\otimes
e_d)\check{\sR} =\begin{cases} v e_c\otimes e_c,
&\text{if $ c=d\le m$,}\\
-v^{-1} e_c\otimes e_c &\text{if $m+1\le c=d$,}\\
(-1)^{\hat c\hat d}  e_d\otimes e_c+(v-v^{-1}) e_c\otimes e_d,
&\text{if
$c>d$,}\\
(-1)^{\hat c \hat d}  e_d\otimes e_c, &\text{if $c<d$.}\\
\end{cases}
\end{equation}
The following result is proved in \cite[Th2.1]{Mit}.

\begin{Lemma}\label{key11} If we define linear operator
$$\check{\sR_i}=\id^{\otimes{i-1}}\otimes \check{\sR}\otimes \id^{r-i-1}: V(m|n)^{\otimes r}\rightarrow V(m|n)^{\otimes r},$$ then
\begin{enumerate} \item[(1)] $(\check{\sR}_i-v) (\check
{\sR}_i+v^{-1})=0$.

\item[(2)] $\check \sR_i \check \sR_j=\check \sR_j \check \sR_i$ if $1\le i<j\le
r-1$.

\item[(3)] $\check \sR_i\check \sR_{i+1}\check \sR_i=\check \sR_{i+1} \check \sR_i
\check \sR_{i+1}$ for any $1\le i\le r-2$.
\end{enumerate}
\end{Lemma}


Consider a new basis for $\sH$ by setting $\sT_w=v^{-l(w)}T_w$.
Then, $\mathcal H$  is an associative $\sZ$-algebra generated by $\sT_i=\up^{-1} T_i, 1\le i\le r-1$
subject to the relations
\begin{equation}\label{heckea}\begin{cases}
 (\sT_i-v) (\sT_i+v^{-1})=0, & \text{for $1\le i\le
r-1$.}\\
  \sT_i \sT_j=\sT_j \sT_i, &\text{for $1\le
i<j\le r-1$,} \\
 \sT_i \sT_{i+1} \sT_i=\sT_{i+1} \sT_i
\sT_{i+1}, &\text{for  $1\le i\le r-2$.}\\
\end{cases}
\end{equation}

 Let
 \begin{equation}\label{Imnr}
I(m|n,r)=\{\mathbf i=(i_1,
i_2, \cdots, i_{r})\in\mathbb N^r\mid 1\le i_j\le m+n\,\forall j\},
\end{equation}
and, for $\mathbf i\in I(m|n,r)$, let
$$e_{\mathbf i}=e_{i_1}\otimes e_{i_2}\otimes \cdots \otimes e_{i_r}.$$
Clearly, the set $\{e_{\mathbf i}\}_{\mathbf i\in I(m|n,r)}$ form a basis for
$V(m|n)^{\otimes r}$.

 For each $\mathbf
i\in I(m|n, r)$,  define $\lambda|\mu\in \Lambda(m|n, r)$ to be the weight $\wt(\mathbf
i)$ of $\mathbf i$ by setting
$$\begin{cases} & \lambda_k =\#
\{k\colon i_j=k, 1\le j\le r\}, \forall 1\le k\le m \\
& \mu_k  =\# \{m+k\colon i_j=m+k, 1\le j\le r\}, \forall 1\le k\le n.\\
\end{cases} $$
 For each $\lambda|\mu\in \Lambda(m|n, r)$, define
$\mathbf i_{\lambda|\mu}\in I(m|n, r)$ by
$$ \bold i_{\lambda|
\mu}=(\underset{\lambda_1}{\underbrace {1, \cdots, 1}}
\cdots,\underset{\lambda_m}{\underbrace { m\cdots, m}},
\underset{\mu_1}{\underbrace {m+1,\cdots, m+1}}\cdots
\underset{\mu_n}{\underbrace {m+n, \cdots, m+n}})$$

The  symmetric group  $\mathfrak S_r$
acts on $I(m|n, r)$ by place permutation:
\begin{equation} \label{rightaction} \mathbf i w=(i_{w(1)},
i_{w(2)}, \cdots, i_{w(r)}).\end{equation}
Clearly, the weight function $\wt$ induces a bijection between
the $\mathfrak S_r$-orbits and $\Lambda(m|n,r)$.

\begin{Prop}\label{iso} The tensor superspace
$V(m|n)^{\otimes r}$ is a right $\mathcal H$-module,
and is isomorphic to the $\sH$-module $\fT(m|n,r)=\bigoplus_{(\lambda, \mu)\in \Lambda} x_\lambda y_\mu \sH$.
\end{Prop}

\begin{proof} By defining an action of  $\sT_i$ on $V(m|n)^{\otimes r}$ via $\check {\sR_i}$, the first assertion follows from Lemma~\ref{key11}.

For any $\mathbf i\in I(m|n, r)$ with $\wt(\mathbf
i)=\lambda|\mu$, we have $\mathbf i=\mathbf i_{\lambda|\mu} d$ where
$d$ is the unique element in $\mathfrak D_{\lambda|\mu}$. Write
$\mathbf i=(i_1, i_2, \cdots, i_r)$ and $(j_1, j_2, \cdots,
j_r)=\mathbf i_{\lambda|\mu}$. Then $i_k=j_{d (k)}$ for all $k$. By definition, we have $\hat
j_k=0$ if $k\le |\lambda|$ and $\hat j_k=1$ if $k> |\lambda|$. Also,
 $j_k\le j_l$ whenever $k\le l$. For any $d\in \mathfrak
D_{\lambda|\mu}$ with $\mathbf i=\mathbf i_{\lambda|\mu} d$, define
\begin{equation}\label{dhat}
  \hat d=\sum_{k=1}^r \sum_{\substack {k<l,\\
i_k>i_l}} \hat i_k \hat i_l.
\end{equation}
Thus, \eqref{rmatrix} implies
\begin{equation}\label{struc0}  (-1)^{\hat d} e_{\bold i}
\sT_k=\begin{cases}(-1)^{\hat d}(-1)^{\hat i_k\hat i_{k+1}} e_{\bold is_k},
& \text{ if }i_k<i_{k+1}; \\
v(-1)^{\hat d} e_{\bold i} , & \text{ if }i_k=i_{k+1}\le m;\\
-v^{-1} (-1)^{\hat d} e_{\bold i},&  \text{ if }i_k=i_{k+1}\ge m+1;\\
(-1)^{\hat d}(-1)^{\hat i_k\hat i_{k+1}} e_{\bold is_k} + (v-v^{-1})(-1)^{\hat d} e_{\bold i}, & \text{ if }i_k>i_{k+1},\\
\end{cases}
\end{equation}
where $s_k=(k,k+1)$. On the other hand,
\begin{equation}\label{struc2}
x_{\lambda } y_{\mu} \sT_d
\sT_k=\begin{cases}x_{\lambda } y_{\mu} \sT_{ds_k},
& \text{ if $d s_k\in \mathfrak
D_{\lambda|\mu}$, }\\
v x_\lambda y_\mu \sT_d, & \text{ if $d s_k=s_l d, s_l\in \mathfrak S_\lambda$,}\\
-v^{-1} x_\lambda y_\mu \sT_d,&  \text{ if $d s_k=s_l d, s_l\in \mathfrak S_\mu$,}\\
x_\lambda y_\mu \sT_{ds_k} + (v-v^{-1})x_\lambda y_\mu \sT_{d}, & \text{ if $ds_k<d$.}\\
\end{cases}
\end{equation}
Since $(-1)^{\hat d}(-1)^{\hat i_k\hat i_{k+1}}=
\widehat{ds_k}$ in the first and last case of \eqref{struc0}, it follows that the $\sZ$-linear map
\begin{equation}\label{VH}
f: V(m|n)^{\otimes r} \rightarrow
\bigoplus_{(\lambda, \mu)\in \Lambda} x_\lambda y_\mu \mathcal H_R:
 (-1)^{\hat d} e_{\bold i_{\la|\mu}d} \mapsto x_\lambda y_\mu \sT_d
 \end{equation}
 is a right $\mathcal H$-module homomorphism.
\end{proof}

\begin{Cor}\label{TensorSSA} There is a superalgebra isomorphism
$$\sS(m|n,r)\cong\text{\rm End}_\sH(V(m|n)^{\otimes r}).$$
\end{Cor}
Hence, the quantum Schur superalgebra defined in \ref{ssch} is the same algebra considered in \cite{Mit}.

\begin{rem}Let $\bfU(m|n)=\bfU_\up^\sigma(\mathfrak{gl}(m,n))$ be the quantum enveloping superalgebra defined in \cite[\S3]{Mit}. Then $\bfU(m|n)$ acts naturally on $\mathbf V(m|n)^{\otimes r}$, where
 $\mathbf V(m|n)^{\otimes r}=V(m|n)^{\otimes r}\otimes\mathbb Q(\up)$. By \cite[Th.~4.4]{Mit}, $\bfU(m|n)$ maps onto the algebra $\text{End}_\bsH(\mathbf V(m|n)^{\otimes r})$. Now, Corollary \ref{TensorSSA} and Theorem \ref{DCP} implies
the Schur--Weyl reciprocity between $\bfU(m|n)$ and $\bsH$ as described in \cite[Th.~4.4]{Mit}.
\end{rem}

\section{Relation with quantum matrix superalgebras}

Like quantum Schur algebras, quantum Schur superalgebras $\sS(m|n,r)$
can also be interpreted as the dual algebra of the $r$th homogeneous
component $\sA(m|n,r)$ of the quantum matrix superalgebra
$\sA(m|n)$. We first recall the following definition which is a
special case of quantum superalgebras with multiparameters  defined
by Manin~\cite[1.2]{Manin}. For simplicity, we assume throughout the
section that $\field$ is a {\it field} of  characteristic
$\text{char}(\field)\neq 2$ and  $v\in \field$.

\begin{Defn}\label{quantumco}  Let $\sA(m|n)$ be the associative
superalgebra over $\field$ generated by  $x_{ij}$, $1\le i, j\le m+n$
subject to the following relations:
\begin{enumerate}
\item[(1)] $ x_{i, j}^2=0$,  for $\hat{i}+ \hat{j}=1$;
\item[(2)] $ x_{ij}x_{ik}=(-1)^{(\hat i+\hat j)(\hat i+ \hat k)}
v^{(-1)^{\hat i+1}} x_{ik}x_{ij}$, for $ j<k$;
\item[(3)] $ x_{ij}x_{kj}=(-1)^{(\hat i+ \hat j)(\hat k+ \hat j)} v^{(-1)^{\hat j+1}} x_{kj}x_{ij}$, for $ i<k$;
\item[(4)]  $ x_{ij}x_{kl}=(-1)^{(\hat i+\hat j)(\hat k+ \hat l)} x_{kl}x_{ij}$, for $i<k$ and
$j>l$;
\item[(5)] $x_{ij}x_{kl}=(-1)^{(\hat i+ \hat j) (\hat k+ \hat l)}
x_{kl}x_{ij}+(-1)^{\hat k \hat j+ \hat k \hat l + \hat j \hat
l}(v^{-1}-v) x_{il}x_{kj}$, for $i<k$ and $j<l$.
\end{enumerate}
\end{Defn}

Manin~\cite{Manin} proved that $\sA(m|n)$ has also a
supercoalgebra structure with comultiplication $\Delta: \sA(m|n)\rightarrow
\sA(m|n)\otimes \sA(m|n)$ and co-unit $\epsilon: \sA(m|n)\rightarrow
\field$ defined by
\begin{equation}\label{comulti}
\Delta(x_{ik})=\sum_{j=1}^{m+n} x_{ij}\otimes x_{jk}, \text{ and }
\epsilon(x_{ij})=\delta_{ij}, \forall 1\le i, j, k\le m+n.
\end{equation}
Further, the $\mathbb Z_2$ grading degree of $x_{ij}$ is $\hat
i+\hat j\in \mathbb Z_2$. The following result is a special case of
~\cite[Th.~1.14]{Manin}.

\begin{thm}\label{coor} Suppose $v^2\neq -1$ in $\field$. Then $\sA(m|n)$ has basis
$$\mathcal B=\left\{\prod_{i,j}
x_{i,j}^{a_{i,j}}\colon a_{ij}\in \mathbb N,\text{ and } a_{ij}\in\{0,1\}
\text{ whenever }\hat i+\hat j=1 \right\},$$ where the order of
$x_{i,j}$ is arranged such that $x_{i,j}$ is the left to $x_{k,l}$ if
either $i<k$ or $i=k$ and $j<l$.
\end{thm}

For each $A=(a_{ij})\in M(m+n)$, define
\begin{equation}\label{matrix} x^A= x_{1,1}^{a_{1,1}}
x_{1,2}^{a_{1,2}}\cdots, x_{1,m+n}^{a_{1,m+n}}
x_{2,1}^{a_{2,1}}\cdots x_{m+n, m+n}^{a_{m+n,m+n}}\end{equation} By
Definition~\ref{quantumco}(a) and Theorem~\ref{coor}, $x^A\not=0$ if and
only if $A\in M(m|n)$. Thus, $\mathcal B=\{x^A\mid A\in M(m|n)\}$.

The bialgebra $\sA(m|n)$ is an $\mathbb N$-graded algebra such that each $x_{ij}$ has
degree $1$. Let $\sA(m|n, r)$ be the subspace of $\sA(m|n)$
spanned by monomials of degree $r$. The following result follows immediately.

\begin{Cor}\label{base} Suppose $v^2\neq -1$ in $\field$. The set
$\mathcal B_r=\{x^A\colon A\in M(m|n, r)\}$ forms an $\field$-basis for the coalgebra $\sA(m|n,r)$.\end{Cor}

We will realize the linear  dual   $\sA(m|n, r)^\ast$ of $A(m|n, r)$
as the endomorphism algebra of the tensor space over the Hecke algebra $\sH_F$
associated to the symmetric group $\mathfrak S_r$. We start by
recalling some notations.

Let $I(m|n,r)$ be the set defined in \eqref{Imnr}.
The group  $\mathfrak S_r$ acts on  $I(m|n, r)\times
I(m|n, r)$ diagonally by  $(\mathbf i,
\mathbf j) w=(\mathbf i w, \mathbf j w)$ for any $w\in \mathfrak
S_r$ and $(\mathbf i, \mathbf j)\in I(m|n, r)\times I(m|n,r)$.
Then there is a bijection between the set of $\fS_r$-orbits and $M(m+n,r)$.
This is seen easily from the map $\jmath$ defined in \eqref{jmath--}:
if $\jmath(\la|\mu,w,\xi|\mu)=A$, where $w\in\fD_{\la|\mu,\xi|\mu}$,
then $A$ is mapped to the orbit containing
$(\mathbf i_{\lambda|\mu}w, \mathbf i_{\xi|\eta})$ or $(\mathbf i_{\lambda|\mu},\mathbf i_{\xi|\eta}w^{-1})$.

Let $x_{\mathbf i, \mathbf j}=x_{i_1, j_1}x_{i_2, j_2}\cdots x_{i_r,
j_r}$.   Since $x_{i,j}$ and $x_{k,l}$  do not commute each other, we
do not have $x_{\mathbf i, \mathbf j}= x_{\mathbf i w, \mathbf j w}$
for  $w\in \mathfrak S_r$, in general.
However, by \cite[8.6,9.6]{DDPW} or a direct argument, we have the
following.

\begin{Lemma}\label{rels}  If $A=(a_{ij})=\jmath(\la|\mu,w,\xi|\eta)\in M(m|n, r)$,
then $$x_{\mathbf i_{\lambda|\mu}, \mathbf i_{\xi|\eta}}^w:=x_{\mathbf i_{\lambda|\mu}, \mathbf i_{\xi|\eta}w^{-1}}=x^A.$$
Moreover,
$x_{\mathbf i_{\xi|\eta}w^{-1},\mathbf i_{\lambda|\mu}}=(-1)^{\widehat{w^{-1}}}x^{A^t},$ where
$\widehat{w^{-1}}$ is defined in \eqref{dhat}.
\end{Lemma}
\begin{proof}
To see the last assertion, note that, if $A=(a_{i,j})$, then
$$x_{\mathbf i_{\xi|\eta}w^{-1},\mathbf i_{\lambda|\mu}}=x_{1,1}^{a_{1,1}}x_{2,1}^{a_{1,2}}\cdots
x_{m+n,1}^{a_{m+n,1}}x_{1,2}^{a_{2,1}}x_{2,2}^{a_{2,2}}\cdots
x_{m+n,2}^{a_{2,m+n}}\cdots\cdots x_{m+n,m+n}^{a_{m+n,m+n,}}.$$ The
assertion  follows from the relation \ref{quantumco}(4).
\end{proof}

Recall from \S3 that we wrote $w_A^-$ for $w$ if
$\jmath(\la|\mu,w,\xi|\mu)=A$. For notational simplicity, we will
write $w_A$ for $w_A^-$ in the rest of the section. Note that $A\in
M(m|n, r)$ if and only if $w_A$ satisfies the trivial intersection
property \ref{TIP}(3):
\begin{equation*}\label{trivialinter}\mathfrak
S_{\lambda^*}\cap w_A\mathfrak S_{^*\!\eta} w_A^{-1}=\{1\}, \text{ and }
\mathfrak S_{{}^*\!\mu}\cap w_A\mathfrak S_{\xi^*} w_A^{-1}=\{1\}.\end{equation*}

Let $\sA(m|n, r)^*$ be the dual space of $\sA(m|n, r)$. It
is well-known that   $\sA(m|n, r)^*$ is a superalgebra with
multiplication given by the following rule
$$
(fg) (v)=(f\otimes g)\Delta(v), \text{ for all $v\in \sA(m|n,
r)^\ast$}.$$ Note that the action of $f\otimes g$ on $\Delta(v)=\sum v_{(1)}\otimes v_{(2)}$ is
given by $$ (f\otimes g)(v_{(1)}\otimes v_{(2)})=(-1)^{ij} f (v_{(1)}) \otimes g
(v_{(2)})$$ if the degree of $g$ (resp. $v_{(1)}$) is $i$ (resp. $j$).

For  $A\in M(m|n, r)$, let  $f_A\in \sA(m|n, r)^*$ be defined by
 $f_A(x^B)=\delta_{A, B}$, for  $B\in M(m|n, r)$. Then $\{f_A\}_{A\in M(m|n, r)}$
 is the dual basis of $\mathcal B_r$.

Since the $\mathbb Z_2$-grading degree of the monomial $ x_{\mathbf i,
\mathbf j}$ is $\sum_{k=1}^{r}( \hat {i}_k + \hat {j}_k)\in \mathbb
Z_2$, it is natural to set the $\mathbb Z_2$-grading degree $\hat {f_A}$  of $f_A$ to be
\begin{equation}\label{grade} \hat {f}_A=\sum_{k=1}^r (\hat i_k+\hat
{j}_k)=|\mu|+|\eta|(\text{mod}\,2)=\hat A,\end{equation} where $\mathbf i=\mathbf i_{\lambda|\mu}$,
 $\mathbf j=\mathbf i_{\xi|\eta } w_A^{-1}$, and $A=\jmath(\la|\mu,w,\xi|\mu)$.


    Manin~\cite{Manin} proved that
the $\field$-space $V(m|n)_\field$, regarded as the specialization
of the $\sZ$-free module $V(m|n)$ in \S8, is a (right)
$\sA(m|n)$-comodule with structure map
$$\delta: V(m|n)_\field\rightarrow
 V(m|n)_\field \otimes \sA(m|n),\ \  e_i\longmapsto \sum_{j}  e_j\otimes x_{j,i}$$
Since  $\sA(m|n)$ is a super-bialgebra,  $V(m|n)_\field^{\otimes r }$ is
also an $\sA(m|n)$-comodule and the structure map is induced by the
structure map $\delta$ on $V(m|n)_\field$. By abuse of notation, we still
use $\delta$ to denote the structure map.  Thus, for
any $e_{\mathbf
i}=e_{i_1}\otimes e_{i_2}\otimes \cdots \otimes e_{i_r}$ with $\mathbf i\in I(m|n, r)$,
\begin{equation}\label{tensor}
\delta(e_{\mathbf i}) =\sum_{\mathbf j\in I(m|n, r)}(-1)^{\sum_{1\le
k<l\le r} \hat j_k( \hat j_l+ \hat i_l)}e_{\mathbf j}\otimes x_{\mathbf j,\mathbf i}. \end{equation}
Restriction makes  $V(m|n)_\field^{\otimes r}$ into an $\sA(m|n,
r)$-comodule, and hence   a left  $\sA(m|n, r)^*$-module  with
the  action given by
$$
f\cdot e =(\id_{V(m|n)_\field^{\otimes r}}\otimes f )\delta (e), \quad
\forall f\in \sA(m|n, r)^*, e\in V(m|n)_\field^{\otimes r}.$$


\begin{Lemma} The action of $\sA(m|n, r)^\ast$ on $V(m|n)_\field^{\otimes r}$  is faithful.\end{Lemma}
\begin{proof}  Suppose $f\cdot e_{\mathbf i}=0$ for all $\mathbf
i\in I(m|n, r)$. By (\ref{tensor}), $f(x_{\mathbf j,\mathbf i})=0$
for all $\mathbf i, \mathbf j\in I(m|n, r)$. In particular,
$f(x_{\mathbf j,\mathbf i_{\lambda|\mu}})=0$ for all
$\lambda|\mu\in \Lambda(m|n, r)$ and  all $ \mathbf j\in I(m|n, r)$.
By Corollary~\ref{base} and Lemma~\ref{rels}, $f=0$.\end{proof}

With the definition of $\hat f_A$, it would be possible to explicitly describe the action $f_A\cdot e_{\mathbf i_{\xi|\eta}}$, and hence, to make a comparison between bases $\{f_A\}$ and $\{\phi_A\}$ under the isomorphism in Theorem \ref{mainquper}.

\begin{Prop}\label{comodule} The linear map    $\check{\sR}: V(m|n)_\field^{\otimes
2}\rightarrow V(m|n)_\field^{\otimes 2}$ defined in \eqref{rmatrix}
is an $\sA(m|n,r)$-comodule homomorphism. Moreover, the actions of
$A(m|n, r)^*$ and $\sH_\field$ on $V(m|n)_\field^{\otimes r}$
commute.
\end{Prop}
\begin{proof} We need verify \begin{equation}\label{comh}\delta\circ\check{\sR}=(\id_{\sA(m|n, r)}\otimes \check{\sR} )\circ \delta
\end{equation}
where $\delta $ is the comodule structure map on $V(m|n)_\field^{\otimes 2}$.
We verify the case $e_i\otimes e_j$ with $i>j$. One can verify the
other cases similarly. We have
$$\begin{aligned} (\id_{\sA(m|n, r)}& \otimes \check{\sR} ) \delta (e_i\otimes e_j)=v
\sum_{k\le m} x_{i,k}x_{j, k} \otimes e_k\otimes e_k\\ & +
(-1)^{\hat j} v^{-1} \sum_{k\ge m+1}  x_{i,k} x_{j,k}  \otimes
e_k\otimes e_k + \sum_{l<k} (-1)^{\hat k \hat j}  x_{ik}x_{jl}
\otimes e_l\otimes e_k\\ & +\sum_{l>k}\left\{ (-1)^{\hat k \hat j}
x_{ik}x_{jl}+(-1)^{\hat l \hat j+ \hat k \hat l}
(v-v^{-1}) x_{il} x_{jk}\right\}\otimes e_l\otimes e_k\\
\end{aligned}. $$
On the other hand,
$$
\begin{aligned}  \delta \circ \check{\sR} (e_i\otimes
e_j)=&\delta((-1)^{\hat i \hat j}  e_j\otimes e_i+(v-v^{-1})
e_i\otimes e_j)
\\
=& (-1)^{\hat i \hat j}  \sum_{k,l} (-1)^{\hat k \hat i+ \hat k \hat
l }x_{jk}x_{il} e_k\otimes e_l \\ & + (v-v^{-1}) \sum_{k, l}
x_{ik}x_{jl}
(-1)^{\hat k \hat j +\hat k \hat l} e_k\otimes e_l\\
\end{aligned}
$$
Comparing the coefficients of $e_k\otimes e_l$ via
Definition~\ref{quantumco}  yields
 $\delta\circ
\check{\sR}(e_i\otimes e_j) =(\id_{\sA(m|n, r)}\otimes \check{\sR} )
\delta (e_i\otimes e_j)$. This proves (\ref{comh}).  Further, it
implies that the  actions of $\sA(m|n, r)^*$ and $\sH_\field$ on
$V(m|n)_\field^{\otimes r}$ commute. (One can also verify it by the
definition of the action of the linear dual of a cosuperalgebra $\sA$
on an $\sA$-cosupermodule. See the definition given in
\cite[p.45]{BK:sergeev}.)
\end{proof}

The following result is the quantum version of \cite[Th.~5.2]{BKu}.

\begin{thm}\label{mainquper} The quantum Schur superalgebra $\sS(m|n,r)_\field$ is isomorphic
to the algebra $\sA(m|n, r)^*$. In other words, we have an algebra isomorphism
 $$\sA(m|n, r)^*\cong \text{\rm End}_{\sH_\field} (\oplus_{\lambda|\mu\in \Lambda(m|n, r) }
x_\lambda y_\mu\sH_\field).$$\end{thm}

\begin{proof} We have already proved that $\sA(m|n, r)^*$ acts
faithfully  on $V(m|n)_\field^{\otimes r}$. So,  $\sA(m|n, r)^\ast$ is a
subalgebra of End$_F(V(m|n)_\field^{\otimes r})$.  By
Proposition \ref{comodule}, $\sA(m|n, r)^\ast$ is a
subalgebra of End$_{\sH_\field} (V(m|n)_\field^{\otimes r})$. A
dimensional comparison (see Theorem \ref{base change} and Corollary \ref{base}) gives the required isomorphism.
\end{proof}

 \providecommand{\bysame}{\leavevmode ---\ }
\providecommand{\og}{``} \providecommand{\fg}{''}
\providecommand{\smfandname}{and}
\providecommand{\smfedsname}{\'eds.}
\providecommand{\smfedname}{\'ed.}
\providecommand{\smfmastersthesisname}{M\'emoire}
\providecommand{\smfphdthesisname}{Th\`ese}

\end{document}